\documentclass[a4paper, 11pt]{article}
\usepackage{amssymb}
\usepackage{t1enc}
\usepackage[english]{babel}
\usepackage{times}
\usepackage{a4wide}
\usepackage{color, hyperref}
\usepackage{amsmath, amsthm}

\newcounter{licznik}[section]

 \newtheorem{theorem}[licznik]{Theorem}
 \newtheorem{lemma}[licznik]{Lemma}
\newtheorem{prop}[licznik]{Proposition}
\newtheorem{cor}[licznik]{Corollary}
 \newtheorem{remark}[licznik]{Remark}

\def\tl{\tilde}

\def\ve{\varepsilon}

\def\te{{\cal T}}

\def\es{{\cal S}}
\def\bes{{\bar \es}}

\def\ef{\mathcal{F}}
\def\ee{\mathbb{E}}
\def\er{\mathbb{R}}
\def\bse{\mathcal{E}}
\def\prob{\mathbb{P}}
\newcommand\ind[1]{1_{#1}}

\makeatletter
\newcommand\assumptionlabel[1]{\hspace\labelsep
                               \normalfont\bfseries #1\ \ \gdef\@currentlabel{#1}}
\newenvironment{assumption}
               {\medskip\list{}{\labelwidth\z@ \itemindent-\leftmargin
                        }}
               {\endlist}
\makeatother

\begin{document}
% Comment out for SIAM format
\title{Impulse control maximising average cost per unit time: a non-uniformly ergodic case\footnote{Research of both authors has been partly supported by NCN grant DEC-2012/07/B/ST1/03298.}}
\author{Jan Palczewski\footnote{School of Mathematics, University of Leeds, LS2 9JT, Leeds, United Kingdom} \and \L{}ukasz Stettner\footnote{Institute of Mathematics Polish Acad. Sci., Sniadeckich 8, 00-656 Warsaw, Poland, and Vistula University}}
\maketitle

\begin{abstract}
This paper studies maximisation of an average-cost-per-unit-time ergodic functional over impulse strategies controlling a Feller-Markov process. The uncontrolled process is assumed to be ergodic but, unlike the extant literature, the convergence to invariant measure does not have to be uniformly geometric in total variation norm; in particular, we allow for non-uniform geometric or polynomial convergence. Cost of an impulse may be unbounded, e.g., proportional to the distance the process is shifted. We show that the optimal value does not depend on the initial point and provide optimal or $\ve$-optimal strategies.
\medskip

\end{abstract}

% SIAM format

% \title{Impulse control maximising average cost per unit time: a non-uniformly ergodic case\thanks{Submitted to the editors on 21 July 2016.\funding{Research of both authors has been partly supported by NCN grant DEC-2012/07/B/ST1/03298.}}}
% \author{Jan Palczewski\thanks{School of Mathematics, University of Leeds, LS2 9JT, Leeds, United Kingdom (\email{J.Palczewski@leeds.ac.uk})} \and \L{}ukasz Stettner\thanks{Institute of Mathematics Polish Acad. Sci., Sniadeckich 8, 00-656 Warsaw, Poland, and Vistula University (email{stettner@impan.pl})}}
% \maketitle
% 
% \begin{abstract}
% This paper studies maximisation of an average-cost-per-unit-time ergodic functional over impulse strategies controlling a Feller-Markov process. The uncontrolled process is assumed to be ergodic but, unlike the extant literature, the convergence to invariant measure does not have to be uniformly geometric in total variation norm; in particular, we allow for non-uniform geometric or polynomial convergence. Cost of an impulse may be unbounded, e.g., proportional to the distance the process is shifted. We show that the optimal value does not depend on the initial state and provide optimal or $\ve$-optimal strategies. 
% \end{abstract}
% 
% \begin{keywords}
%  impulse control, ergodic control, non-uniformly ergodic Markov process, unbounded cost
% \end{keywords}
% 
% \begin{AMS}
% 93E20,  60J25
% \end{AMS}

\section{Introduction}\label{sec:intro}
Let $(X_t)$ be a Feller-Markov process on $(\Omega, F,(F_t))$ with values in a locally compact space $E$ with the metric $\rho$ and Borel $\sigma$ field ${\cal E}$. The process starting from $x$ at time $0$ generates a probability measure $\prob^x$; $\ee^x$ denotes a related expectation operator. Process $(X_t)$ is controlled by impulses $(\tau, \xi)$: at time $\tau$ the process is shifted from the state $X_{\tau}$ to the state $\xi$ at the cost of $c(X_\tau, \xi)$ and follows its dynamics until the next impulse. We assume that impulses shift the process to a compact set $U \subseteq E$, i.e., $\xi \in U$ and the cost function $c$ is negative, continuous and uniformly bounded away from zero.\footnote{In a slightly misleading way, we call $c$ the cost. As it stands in functional \eqref{eqn:main_functional} with a plus sign, it is assumed to be negative and bounded away from zero, i.e., there is a minimum cost of an impulse.} A strategy $V=(\tau_i, \xi_i)$ is \emph{admissible} for $x \in E$ if $\tau_i$ form an increasing sequence of stopping times (possibly taking the value $\infty$) with $\lim_{i \to \infty} \tau_i = \infty$, $\prob^x$-a.s. To describe the evolution of the controlled process we introduce a construction of \cite[Section 2]{Stettner1983a} which follows ideas of \cite{Robin1978}. Namely, we consider $\Omega=D(R^+,E)^\infty$, where $D(R^+,E)$ is a canonical space of right continuous, left limited functions on $R^+$ taking values in $E$. We assume that $(F^1_t)$ is a canonical filtration on $D(R^+,E)$  and inductively define $F^{n+1}_t=F^n_t \otimes F_t$. The stopping times  $\tau_i$ are adapted $F^i_t \times \left\{\emptyset,D(R^+,E)\right\}^\infty$ while the impulses $\xi_i$ to $F^i_{\tau_i} \times \left\{\emptyset,D(R^+,E)\right\}^\infty$.
The trajectory of the controlled process $(X_t)$ is defined using coordinates $x^n$ of the canonical space $\Omega$, i.e. $X_t=x^n_t$ for $t\in [\tau_{n-1},\tau_n)$, with $\tau_0=0$. Given an impulse strategy $V$ following \cite[Section 2]{Stettner1983a} and \cite[Chapter 5 and Appendix 2]{Robin1978} we define a probability measure $\prob$ on $\Omega$. Although the controlled process $(X_t)$ and probability measure $\prob$ depend on the control strategy $V$ in what follows we shall not indicate that explicitly.

Our goal is to maximize over all admissible strategies the functional
\begin{equation}\label{eqn:main_functional}
J\big(x, (\tau_i, \xi_i)\big) = \liminf_{T \to \infty} \frac{1}{T} \ee^x \Big\{ \int_0^T f(X_s) ds + \sum_{i = 1}^\infty \ind{\tau_i \le T} c(X_{\tau_i-}, \xi_i) \Big\},
\end{equation}
where $f$ is a continuous bounded function and $X_{\tau_i-}$ is the state of the process before the $i$-th impulse with a natural meaning if there is more than one impulse at the same time. Alternatively, we shall also consider a weaker form of \eqref{eqn:main_functional}, namely
\begin{equation}\label{eqn:weaker_functional}
\hat{J}\big(x, (\tau_i, \xi_i)\big) = \liminf_{n \to \infty} \frac{1}{\ee^x\{\tau_n\}} \ee^x \Big\{ \int_0^{\tau_n} f(X_s) ds + \sum_{i = 1}^n \ind{\tau_i \le T} c(X_{\tau_i-}, \xi_i) \Big\},
\end{equation}
assuming that $(\tau_n)$ are such that $\ee^x\{\tau_n\}<\infty$.

Controlling random systems by impulses, i.e., discrete interventions, is often the only feasible strategy from an application point of view and, therefore, the literature is extensive. For applications in finance, the reader is referred to \cite{DPDS2010, Kelly2011} and references therein. Intensive studies of impulse control of diffusions and diffusions with jumps are presented in \cite{bensoussan1984}. Impulse control of Markov processes with average cost per unit time criterion \eqref{eqn:main_functional} has been studied first in \cite{Robin1981, Robin1983} under uniform ergodicity assumption for constant cost for impulses. These results were extended  to a separated cost (for definition see Proposition \ref{propseparated}) in \cite{Stettner1983a} and to quasicompact transition semigroups in \cite{Stettner1986b}. The problem was also studied under some compactness assumptions in \cite{GatarekStettner}. Ergodic impulse control of diffusion processes on bounded domains was studied in \cite{LionsPerthame1986} and \cite{Perthame1988}. The extension to unbounded domains although in $\er$ only, with linear impulse cost function $c$ depending on the size of an impulse and with $f \le 0$ was tackled in \cite{Jack2006}. Average cost per unit time functionals have also been widely studied in a different setting where the control affects diffusion process continuously, see the monograph \cite{Arapostathis2012} for a detailed discussion.

Solution of problems of the form \eqref{eqn:main_functional} and \eqref{eqn:weaker_functional} usually follows through a study of an auxiliary Bellman equation
\begin{equation}\label{eq1}
w(x)=\sup_\tau \liminf_{T \to \infty} \ee^x \bigg\{\int_0^{\tau \wedge T} (f(X_s)-\lambda)ds + Mw(X_{\tau \wedge T})\bigg\},
\end{equation}
where $Mw(x)=\sup_{\xi \in U} [w(\xi) + c(x,\xi)]$. The solution is a pair: a function $w:E \to \er$ and a constant $\lambda$. One of the main contributions of this paper is showing that when the process is not uniformly ergodic or the cost function $c$ is unbounded, \eqref{eq1} has a solution. The function $w$, which we will often call the value function, is unbounded as is $Mw$.  We prove that the constant $\lambda$ in the solution to the Bellman equation \eqref{eq1} is an optimal value for the functional $\hat{J}$ and frequently also for $J$,  while an optimal stopping time for \eqref{eq1} provides times of consecutive impulses in the optimal strategy. The impulses themselves are given by the maximiser of $Mw(x)$ which is shown to depend continuously on $x$ and, therefore, is measurable.

The novelty of this paper is that
\begin{itemize}
\item the uncontrolled process $(X_t)$ is not uniformly ergodic,% which leads to an unbounded value function for the auxiliary Bellman equation \eqref{eq1},
\item the cost function is not bounded, hence, it can measure the size of an impulse using the distance between the state before and after the impulse.
\end{itemize}
To the best of our knowledge, this paper is the first significant extension of a general theory of ergodic impulse control of Feller-Markov processes since 1980s. The relaxation of uniform ergodicity opens up the theory applicable to many ergodic processes encountered in applications, including an Ornstein-Uhlenbeck process with Levy noise.

% Closely related to \eqref{eq1} is an undiscounted stopping problem
% \[
% v(x) = \sup_\tau \liminf_{T \to \infty} \ee^x \left\{ \int_0^{\tau \wedge T} f(X_s) ds + g(X_{\tau \wedge T}) \right\},
% \]
% which is of independent interest. Undiscounted optimal stopping problems for $f, g \ge 0$ were studied in \cite{Morimoto1991, Morimoto1994}, while for $f \le 0$ are covered in \cite{shiryaev1978}. Reference \cite{palczewski2014} researches the case of continuous bounded $f$ and $g$. Here, we relax the boundedness assumption on $g$ and demonstrate the continuity of the value function and the form of optimal strategies under weak assumptions on the process $(g(X_t))$. Our main assumption is that $\mu(f) < 0$, where $\mu$ is the invariant measure of $(X_t)$, which encourages early stopping similarly as discounting does in a classical case (for details see \cite[Section 2]{palczewski2014}. This is of paramount importance for the study of impulse control under non-uniformly ergodic assumption because $Mw$ is unbounded in \eqref{eq1}.

The paper is structured as follows. Section 2 provides preliminary results for $\alpha$-potentials of centred $f$. %Sections 3-4 are devoted to the study of undiscounted stopping problems with unbounded terminal cost for finite and infinite horizon, respectively. 
In Section 3 we address the  impulse control problem with average cost per unit time functional \eqref{eqn:main_functional} and \eqref{eqn:weaker_functional} and an unbounded cost for a non-uniformly ergodic underlying process under an assumption that the zero-potential of centred $f$ is bounded from below. This restriction is relaxed in Section 4, where, using approximation techniques, we show that the optimal values in  \eqref{eqn:main_functional} and \eqref{eqn:weaker_functional} do not depend on $x$. We also provide $\ve$-optimal strategies through solutions to auxiliary impulse control problems that satisfies assumptions of Section 3.

\section{Preliminaries}
We write $P_t$ for the semigroup, acting on bounded Borel functions, corresponding the (uncontrolled) Markov process $(X_t)$: $P_t \phi(x) = \ee^x \{ \phi(X_t) \}$. A transition probability measure is denoted by $P_t(x, \cdot) := \prob^x \{ X_t \in \cdot \}$. We make the following assumptions:
\begin{assumption}
\item[(A1)] \label{ass:weak_feller}
(Weak Feller property)
$$
P_t\, \mathcal{C}_0 \subseteq \mathcal{C}_0,
$$
where $\mathcal C_0$ is the space of continuous bounded functions $E \to \er$ vanishing in infinity.
\end{assumption}
\begin{assumption}

\item[(A2)] \label{ass:speed_ergodic} There is a unique probability measure $\mu$ on $\bse$, a function $K: E \to (0, \infty)$ bounded on compacts and a function $h:[0, \infty) \to \er_+$ such that $\int_0^\infty h(t) dt < \infty$ and for any $x \in E$
\[
\| P_t(x, \cdot) - \mu(\cdot) \|_{TV} \leq K(x) h(t),
\]
where $\| \cdot \|_{TV}$ denotes the total variation norm. Furthermore, $\ee^x\left\{K(X_T)\right\}<\infty$ for each $T\geq 0$, and for any compact set $\Gamma\subset E$ and a sequence of sets $A_T \in \ef_T$
\[
\lim_{T\to \infty}\sup_{x\in \Gamma} \prob^x\left\{A_T\right\}= 0 \implies \lim_{T\to \infty}\sup_{x\in \Gamma}\ee^x \{\ind{A_T} K(X_T) \}= 0.
\]
%$, whenever for a sequence $A_T\in F_T$ we have $\lim_{T\to \infty}\sup_{x\in \Gamma}P_x\left\{A_T\right\}= 0$ and .
\end{assumption}

Assumption \ref{ass:weak_feller} is necessary to establish the existence of optimal stopping times for general weak Feller processes (a counter-example when it is relaxed is provided at the end of Section 3.1 in \cite{Palczewski2008}). The class of weakly Feller processes \ref{ass:weak_feller} comprises Levy processes \cite[Theorem 3.1.9]{Applebaum2004}, solutions to stochastic differential equations with continuous coefficients driven by Levy processes (see, e.g., \cite[Theorem 6.7.2]{Applebaum2004}).

The first part of Assumption \ref{ass:speed_ergodic} satisfied by non-uniform geometrically ergodic or polynomially ergodic processes with examples discussed in \cite[Section 6]{palczewski2014}.
The second part of Assumption \ref{ass:speed_ergodic} is weaker than requiring that random variables $\left\{K(X_T), T\geq 0\right\}$ are uniformly integrable for initial states $x$ of $(X_t)$ from compact sets. However, the following condition which implies uniform integrability is more explicit to verify: $\sup_{x\in \Gamma} \sup_{T\geq 0}\ee^x \{K(X_T)^{1+\beta} \}<\infty$, for any compact set $\Gamma$ and some $\beta>0$ possibly depending on $\Gamma$. It is easy to verify using this condition that a standard Ornstein-Uhlenbeck process satisfies Assumption \ref{ass:speed_ergodic}.

\begin{lemma}
Under \ref{ass:weak_feller} the operator $P_t$ transforms continuous bounded from above functions into upper semi continuous functions bounded from above.
\end{lemma}
\begin{proof}
By \cite[Corollary 2.2]{Palczewski2008} the semigroup $P_t$ transforms continuous bounded functions into continuous bounded functions. Approximating a continuous function $\varphi$ bounded from above by a sequence of bounded functions $\varphi_n = \max (\varphi,  -n)$ and applying Fatou's lemma completes the proof.
\end{proof}

Let, for $\alpha\ge0$
\begin{equation}\label{eqalphap}
q_\alpha(x)=\ee^x\left\{\int_0^\infty e^{-\alpha t} (f(X_t)-\mu(f))dt\right\}
\end{equation}
with $q := q_0$. We have
\begin{lemma}\label{lq}
Under \ref{ass:weak_feller} and \ref{ass:speed_ergodic} we have that $q_\alpha(x) \to q(x)$ uniformly on compact sets as $\alpha \to 0$, and $q$ is a continuous function such that for any bounded stopping time $\tau$
\begin{equation}\label{eq0potent}
q(x)=\ee^x\left\{\int_0^\tau(f(X_t)-\mu(f))dt + q(X_\tau)\right\}.
\end{equation}
Moreover, for any compact set $\Gamma\subset E$ and a sequence of sets $A_T \in \ef_T$ we have
\begin{equation}\label{eqn:eq0potent_AT}
\lim_{T\to \infty}\sup_{x\in \Gamma} \prob^x\left\{A_T\right\}= 0 \implies \lim_{T\to \infty}\sup_{x\in \Gamma}\sup_{\alpha \in [0,1)} \ee^x \{\ind{A_T} |q_\alpha(X_T)| \}= 0.
\end{equation}
\end{lemma}
\begin{proof}
By \ref{ass:speed_ergodic} we have that $|q_\alpha(x)|\leq K(x) \|f\| \int_0^\infty h(t)dt$ for $\alpha \in [0,1)$, where $\|\cdot\|$ is the supremum norm, so $q_\alpha(x)$ is well defined. Now
\begin{equation}\label{eqo}
|q(x)-q_\alpha(x)|\leq \int_0^\infty (1-e^{-\alpha t})|P_t(f-\mu(f))(x)|dt\leq \int_0^\infty (1-e^{-\alpha t}) K(x) \|f\| h(t)dt \to 0
\end{equation}
as $\alpha \to 0$ uniformly on compact sets, because $K(x)$ is bounded on compact sets. Consequently, since under \ref{ass:weak_feller} $q_\alpha$ is a continuous function, we have that $q$ is also continuous. We have
\begin{equation*}%\label{eqosz}
 \sup_{x\in \Gamma}\sup_{\alpha \in [0,1)}\ee^x \{\ind{A_T} |q_\alpha(X_T)| \}\leq   \sup_{x\in \Gamma}\ee^x \{\ind{A_T} K(X_T)\}\|f\|\int_0^\infty h(t)dt \to 0
 \end{equation*}
as $T\to \infty$ provided that $\lim_{T\to \infty}\sup_{x\in \Gamma}P_x\left\{A_T\right\}= 0$. It remains to show \eqref{eq0potent}. For $\alpha>0$ and $T\geq 0$ we have
\(
q_\alpha(x)=\ee^x\left\{\int_0^T e^{-\alpha t}(f(X_t)-\mu(f))dt + e^{-\alpha T}q_\alpha(X_T)\right\}.
\)
Easily, 
\[
\lim_{\alpha \to 0} \ee^x\big\{\int_0^T e^{-\alpha t}(f(X_t)-\mu(f))dt \big\} = \ee^x\big\{\int_0^T (f(X_t)-\mu(f))dt\big\}.
\]
Denoting $L=\|f\| \int_0^\infty h(t)dt$ and using \ref{ass:speed_ergodic} we obtain
\begin{align*}
&\big|\ee^x\left\{e^{-\alpha t}q_\alpha(X_t)\right\}-\ee^x\left\{q(X_t)\right\}\big|\\[2pt]
&\leq (1-e^{-\alpha t})\ee^x \left\{|q_\alpha(X_t)|\right\} +  \ee^x \left\{|q_\alpha(X_t)-q(X_t)|\right\}\\[2pt]
&\leq (1-e^{-\alpha t}) L\, \ee^x \{K(X_t)\} +\ee^x \left\{\ind{\rho(x,X_t) < R}|q_\alpha(X_t)-q(X_t)|\right\}+ \ee^x \left\{\ind{\rho(x,X_t) \ge R}L K(X_t)\right\}\\
&=a_\alpha+b_\alpha+c_R.
\end{align*}
Clearly $a_\alpha \to 0$ as $\alpha \to 0$. By \eqref{eqo} also $\lim_{\alpha \to 0} b_\alpha =  0$ for any fixed $R$. By \ref{ass:weak_feller} and \cite[Proposition 2.1]{Palczewski2008} taking into account integrability of $K(X_t)$ we obtain that $c_R\to 0$ as $R \to \infty$. Consequently, $q(X_t)$ is integrable and
\(
q(x)=\ee^x\big\{\int_0^T (f(X_t)-\mu(f))dt + q(X_T)\big\},
\)
from which it follows that $Z_s=\int_0^s (f(X_t)-\mu(f))dt + q(X_s)$ is a martingale and we immediately have \eqref{eq0potent}.
\end{proof}

%\section{Impulse control with shifts to a positively ergodic compact set $U$}\label{sec:impulse_control}
\section{Optimal control when \texorpdfstring{$q$}{q} is bounded from below}\label{sec:impulse_control}

We make the following standing assumption for the cost function $c$:
\begin{assumption}
\item[(B1)] \label{ass:cost_function}There is $c < 0$ such that $c(x,x') \le c$ for $(x,x')\in E \times U$, and for $z, z' \in U$
\begin{equation}\label{triangle}
c(x,z)\geq c(x,z')+c(z',z).
\end{equation}
\end{assumption}
Define $\underline{c}(x)=\inf_{a\in U}c(x,a)$ and $\bar{c}(x)=\sup_{a\in U}c(x,a)$. Denote by $\es$ the family of stopping times taking finite values only and by $\bes$ the extension of the latter to stopping times with possibly infinite values. We follow a convention that, unless specified otherwise, all stopping times are from $\bes$.

We will follow a vanishing discount approach, see e.g. \cite{Robin1981}. We consider first a discounted cost impulse control problem which consists in maximization of the functional
\begin{equation}\label{eqn:disc_functional}
J_\alpha\big(x, (\tau_i, \xi_i)\big) = \ee^x \Big\{ \int_0^\infty e^{-\alpha s} f(X_s) ds + \sum_{i = 1}^\infty \ind{\tau_i < \infty} e^{-\alpha \tau_i} c(X_{\tau_i-}, \xi_i) \Big\}
\end{equation}
over admissible impulse strategies\footnote{Recall that $(\tau_i, \xi_i)$ is an admissible strategy if $(\tau_i)$ is a non-decreasing sequence of stopping times from $\bes$ and $\xi_i \in U$ are $\ef_{\tau_i}$-measurable random variables. For more details including construction of the controlled process, see the introduction.} $(\tau_i, \xi_i)$ with the optimal value denoted by $v_\alpha (x)$. Using $v_\alpha$ we will then obtain a sequence of functions converging, as $\alpha \to 0$, to $w$ in \eqref{eq1}. From there, we will derive an optimal value and an optimal strategy for \eqref{eqn:main_functional}.

The following assumption is used for characterisation of the value function $v_\alpha$ as a fixed point of an appropriate Bellman operator:
\begin{assumption}
\item[(B2)]\label{ass:unifintcalpha} For any compact set $\Gamma \subset E$ and any $T > 0$, the random variable $\zeta_T = \sup_{t \in [0, T]} |\underline{c}(X_t)|$ is uniformly integrable with respect to $\prob^y$ for $y \in \Gamma$, i.e.,
\[
\lim_{n \to \infty} \sup_{y \in \Gamma} \ee^y \{ \zeta_T \ind{\zeta_T > n} \} = 0.
\]
\end{assumption}

For a continuous function $v$, consider an operator
\begin{equation}\label{def:oper1}
\te v(x):= \sup_{\tau \in \bes} \ee^x\Big\{\int_0^{\tau} e^{-\alpha s} f(X_s)ds + \ind{\tau < \infty} e^{-\alpha \tau}Mv(X_{\tau})\Big\},
\end{equation}
where $Mv(x):=\sup_{\xi\in U}[c(x,\xi)+v(\xi)]$ and its approximation
\[
\te_L v(x):= \sup_{\tau \in \bes} \ee^x\Big\{\int_0^{\tau} e^{-\alpha s} f(X_s)ds + e^{-\alpha \tau}M_L v(X_{\tau})\Big\}
\]
with $M_Lv(x):=\sup_{\xi\in U}[c(x,\xi) \vee (-L) +v(\xi)]$. In the definition of operator $\te_L$, the indicator $\ind{\tau < \infty}$ is omitted intentionally as $M_L v$ is bounded, so for infinite value of $\tau$ the discounting makes that term equal $0$.

Define a functional with a truncated cost function
\begin{equation}\label{eqn:disc_functional_L}
J^L_\alpha\big(x, (\tau_i, \xi_i)\big) = \ee^x \Big\{ \int_0^\infty e^{-\alpha s} f(X_s) ds + \sum_{i = 1}^\infty e^{-\alpha \tau_i} \big(c(X_{\tau_i-}, \xi_i) \vee (-L) \big)\Big\},
\end{equation}
and denote its optimal value by $v^L_\alpha(x)$.

\begin{lemma}\label{lem:disc_approx}
Assume \ref{ass:unifintcalpha} and that $v$ is a continuous function with $\| v\| \le \|f\| / \alpha$. For each $n \ge 1$, the limits
\[
\lim_{L \to \infty} \te^n_L v(x) = \te^n v(x), \qquad \lim_{L \to \infty} M_L \te^n_L v(x) = M \te^n v(x)
\]
are uniform in $x$ from compact sets.
\end{lemma}
\begin{proof}
Define
\( \te_T v(x) = \sup_{\tau} \ee^x \big\{ \int_0^{\tau \wedge T} e^{-\alpha s} f(X_s) ds + \ind{\tau \le T} e^{-\alpha \tau} Mv(X_\tau) \big\}. \)
Take any stopping time $\tau$ and let $\hat \tau = \tau \ind{\tau \le T} + \infty \ind{\tau > T} \in \bes$. Then $\hat \tau$ brings up the same value of the functional for $\te$ as $\tau$ brings in functional for $\te_T$. Hence $\te v \ge \te_T v$. Due to the boundedness of $Mv(x)$ from above by $\|f\|/\alpha$, we have for any $T > 0$
\begin{align*}
\te v(x) - \te_T v(x) 
&\le \sup_\tau \ee^x \Big\{ \int_{\tau \wedge T}^\tau e^{-\alpha s} f(X_s) ds + \ind{T < \tau < \infty} e^{-\alpha \tau} Mv(X_\tau) \Big\}\\
%&\le \sup_\tau \ee^x \Big\{ \int_{\tau \wedge T}^\tau e^{-\alpha s} f(X_s) ds + \ind{T < \tau < \infty} e^{-\alpha \tau} \sup_{\xi \in U} v(\xi) \Big\}\\
% &\le \sup_\tau \ee^x \Big\{ \int_{\tau \wedge T}^\tau e^{-\alpha s} \|f\| ds + \ind{T < \tau < \infty} e^{-\alpha \tau} \frac{\|f\|}{\alpha} \Big\}\\
 &\le \sup_\tau \ee^x \Big\{ \frac{\|f\|}{\alpha} \big(e^{-\alpha (\tau \wedge T)} - e^{-\alpha \tau}\big) + \ind{T < \tau < \infty} e^{-\alpha \tau} \frac{\|f\|}{\alpha} \Big\}\\
 &= \sup_\tau \ee^x \Big\{\ind{\tau > T} e^{-\alpha T} \frac{\|f\|}{\alpha} \Big\} = e^{-\alpha T} \frac{\|f\|}{\alpha}.
\end{align*}
Similarly,
\[
0 \le  \te_L v(x) - \sup_{\tau} \ee^x \Big\{ \int_0^{\tau \wedge T} e^{-\alpha s} f(X_s) ds + \ind{\tau \le T} e^{-\alpha \tau} M_Lv(X_\tau) \Big\} \le e^{-\alpha T} \frac{\|f\|}{\alpha}.
\]
Hence,
\begin{align*}
0 \le \te_L v(x) - \te v(x) 
&\le \sup_{\tau} \ee^x \big\{ \ind{\tau \le T} e^{-\alpha \tau} \big(Mv(X_{\tau}) - M_L v(X_{\tau}) \big) \big\} + 2  e^{-\alpha T} \frac{\|f\|}{\alpha}\\ 
&\le \ee^x \{ \zeta_T \ind{\zeta_T > L} \} + 2  e^{-\alpha T} \frac{\|f\|}{\alpha},
\end{align*}
where $\zeta_T$ is defined in assumption \ref{ass:unifintcalpha}. The second term can be made arbitrarily small by choosing $T$ sufficiently large. The first term converges to $0$ as $L \to \infty$ uniformly in $x$ from compact sets by \ref{ass:unifintcalpha}. Hence, $\te_L v(x)$ converges, as $L \to \infty$, to $\te v(x)$ uniformly in $x$ from compact sets. Then, $\lim_{L \to \infty} M_L\te_L v(x) = M \te v(x)$ uniformly on compacts. Proceeding by induction and using arguments similar to those above, the proof of the lemma is completed.
\end{proof}

\begin{lemma}
Assume \ref{ass:unifintcalpha} and that $v$ is a continuous function with $\| v\| \le \|f\| / \alpha$. Then in \eqref{def:oper1} the supremum can be restricted to finite stopping times:
\begin{equation}\label{def:oper}
\te v(x)= \sup_{\tau \in \es} \ee^x\left\{\int_0^{\tau} e^{-\alpha s} f(X_s)ds + e^{-\alpha \tau}Mv(X_{\tau})\right\}.
\end{equation}
\end{lemma}
\begin{proof}
From the proof of Lemma \ref{lem:disc_approx}, $\ve$-optimal stopping times $\tau$ for $\te v(x)$ take values in $[0, T] \cup \{\infty\}$ for some $T$ depending on $\ve$. Under assumption \ref{ass:speed_ergodic} there is a compact set $K\subset E$ with $\mu(K) > 0$, so it is recurrent. Define $\sigma_1 = \inf\{ t \ge T:\ X_t \in K\}$ and $\sigma_{n+1} = \inf\{ t \ge \sigma_n + 1:\ X_t \in K\}$. Then $\sigma_n < \infty$ and $\lim_{n \to \infty} \sigma_n = \infty$. Set $\tau_n = \tau \wedge \sigma_n$. The boundedness of $Mv$ on $K$ and boundedness of $f$ yield then
\[
\ee^x\left\{\int_0^{\tau} e^{-\alpha s} f(X_s)ds + \ind{\tau < \infty}e^{-\alpha \tau}Mv(X_{\tau})\right\} = \lim_{n \to \infty} \ee^x\left\{\int_0^{\tau_n} e^{-\alpha s} f(X_s)ds + e^{-\alpha \tau_n} Mv(X_{\tau_n})\right\}.
\]
\end{proof}

The finding of the above lemma that the supremum in \eqref{def:oper1} can be restricted to finite stopping times will be used implicitely in the proof of Theorem \ref{disc prob}.

%\jp{LUKASZU: w paru miejscach ponizej dodalem $\ind{\tau < \infty}$ gdyz $Mv_\alpha$, $Mw_\alpha$ sa nieograniczone i nie mozna zakladac, ze dyskonto zabije ich wartosci. Jest to przez to konieczne, zeby $v_\alpha$ bylo optymalna wartoscia $J_\alpha$ (np. gdy optymalna strategia ma element nie rob nic). Pozniej, zalozenie o powracaniu do zbioru zwartego (na mocy A2) oraz aproksymacja przez $\te_L$ pozwoli nam na pozbycie sie tego indykatora.}
\begin{theorem}\label{disc prob}
Under the assumptions  \ref{ass:weak_feller} and \ref{ass:unifintcalpha},  the function $v_\alpha$ is continuous and it is a solution to the equation
\begin{equation}\label{eq9}
v_\alpha(x)=\sup_{\tau \in \es} \ee^x\left\{\int_0^{\tau} e^{-\alpha s} f(X_s)ds + e^{-\alpha \tau}Mv_\alpha(X_{\tau})\right\},
\end{equation}
where $Mv(x):=\sup_{\xi\in U}[c(x,\xi)+v(\xi)]$. Furthermore, $|v_\alpha|\leq \frac{\|f\|}{\alpha}$ and it is approximated by $v^L_\alpha(x)$ uniformly in $x$ from compact sets.
\end{theorem}
\begin{proof}
Without loss of generality we can assume that $f \ge 0$ in \eqref{eqn:disc_functional}. Notice also that if $\| v\| \le \|f\| / \alpha$, then $\| \te v\| \le \|f\| / \alpha$. Let $r (x) =\ee^x\left\{\int _0^\infty e^{-\alpha s} f(X_s)ds\right\}$ be the resolvent of $f$. The sequence ${\cal T}^n r(x)$ is nondecreasing and bounded, therefore converges to a fixed point of the equation \eqref{eq9}. Thanks to the boundedness of the functional $J^L_\alpha$, classical results yield that the function $\te^n_L r(x)$ is continuous. By Lemma \ref{lem:disc_approx}, $\te^n_L r(x) \to \te^n r(x)$ as $L \to \infty$ uniformly in $x$ from compact sets, which implies the continuity of $\te^n r(x)$. Using standard supermartingale arguments of Theorem V.2.1 and Lemma II.2.2 in \cite{Robin1978} one can show that ${\cal T}_L^n r(x)$ corresponds to the optimal value of the functional $J_\alpha^L\big(x,(\tau_i,\xi_i)\big)$ over impulse strategies consisting of at most $n$ impulses.  For a fixed strategy $(\tau_i, \xi_i)$ monotone convergence theorem implies $\lim_{L \to \infty} J_\alpha^L\big(x,(\tau_i,\xi_i)\big) = J_\alpha\big(x,(\tau_i,\xi_i)\big)$. Hence, ${\cal T}^n r(x)$ is the optimal value of the functional $J_\alpha\big(x,(\tau_i,\xi_i)\big)$ for strategies restricted to at most $n$ impulses.

For $\ve > 0$, let $V_\ve$ be an $\ve$-optimal strategy for $v_\alpha(x)$. Denote by $N_T$ the number of impulses of this strategy up to and including time $T$. Then
\[
-\frac{\|f\|}{\alpha}-\ve \le J_\alpha\big(x,V_\ve \big)\le \frac{\|f\|}{\alpha} + e^{-\alpha T} c \ee^x\left\{N_T\right\},
\]
from which it follows that
%\begin{equation}\label{impest}
\( \ee^x\left\{N_T\right\}\leq \frac{e^{\alpha T}}{-c}\left( 2\frac{\|f\|}{\alpha} + \ve \right). \)
 %\end{equation}
Denote by $V_{\ve,n}$ the strategy $V_\ve$ restricted to $n$ impulses. For $T>0$ using the above bound for $\ee^x \{ N_T \}$, we obtain
\begin{align*}
\big|J_\alpha\big(x,V_\ve \big)-J_\alpha \big(x,V_{\ve,n} \big)\big|
&\le  2 \frac{\|f\|}{\alpha}\ee^x \left\{e^{-\alpha \tau_{n+1}}\right\}
\le  2 \frac{\|f\|}{\alpha} \ee^x \left\{ e^{-\alpha T} \ind{\tau_{n+1} > T} + \ind{\tau_{n+1} \le T}\right\} \\
&\le 2 \frac{\|f\|}{\alpha}\left(e^{-\alpha T}+\prob^x \left\{N_T \ge n+1\right\}\right)\\
&\le 2 \frac{\|f\|}{\alpha}\left(e^{-\alpha T}+ \frac{e^{\alpha T}}{-(n+1)c}\left(2\frac{\|f\|}{\alpha} + \ve \right)\right).
\end{align*}
Since the right-hand side does not depend on $x$, letting $n\to \infty$ then $T\to \infty$ and taking into account the $\ve$-optimality of $V_\ve$ we have that ${\cal T}^n r(x)$ converge uniformly (in $x\in E$) to $v_\alpha(x)$. Identically, we prove $\lim_{n \to \infty} \sup_{L \ge -c} |\te^n_L r(x) - v^L_\alpha(x)| = 0$ uniformly in $x \in E$. This, together with assertions of Lemma \ref{lem:disc_approx}, implies that $v^L_\alpha(x) \to v_\alpha(x)$ uniformly in $x$ from compact sets.
\end{proof}
\begin{remark}
In the case when $\underline{c}$ is bounded the assertions of Theorem \ref{disc prob} follow directly from \cite{Robin1978} or \cite{Stettner1983a}.
\end{remark}

Fix $z\in U$. It will be an anchor point for further definition of functions $w_\alpha$. We have the following bounds for $v_\alpha$ and for the difference $v_\alpha(x) - v_\alpha(z)$.
\begin{lemma}\label{oszac}
We have $v_\alpha(x)\geq c(x,z)+v_\alpha(z)$ for $x \in E$ and
\begin{equation}
c(x,z)\leq v_\alpha(x)-v_\alpha(z)\leq -c(z,x) \quad \text{ for $x \in U$}. \label{eq10}
\end{equation}
\end{lemma}
\begin{proof} Clearly, $v_\alpha(x)\geq Mv_\alpha(x)\geq c(x,z)+v_\alpha(z)$. Whenever $x\in U$ we also have $v_\alpha(z)\geq Mv_\alpha(z)\geq c(z,x)+v_\alpha(x)$.
\end{proof}

Define $w_\alpha(x)=v_\alpha(x)-v_\alpha(z)$ for $x\in E$. We deduce from Lemma \ref{oszac} a bound on $w_\alpha$ on $U$ which is independent of $\alpha$:
\begin{equation}\label{eq11}
\sup_{x\in U} |w_\alpha(x)|\leq \sup_{x\in U}\left\{|c(x,z)|\vee |c(z,x)|\right\}:=\kappa.
\end{equation}
From \eqref{eq9} we obtain easily the following equation for $w_\alpha$
\begin{equation}\label{eq12}
w_\alpha(x)=\sup_{\tau \in \es} \ee^x\left\{\int_0^{\tau} e^{-\alpha s} (f(X_s)-\alpha v_\alpha(z))ds + e^{-\alpha \tau}Mw_\alpha(X_{\tau})\right\}.
\end{equation}

Define $D_U=\inf \left\{s\geq 0: X_s\in U\right\}$ and $t(x)=\ee^x\left\{D_U\right\}$. We make the following assumption
\begin{assumption}
\item[(B3)]\label{ass:positrec} For any compact set $\Gamma\subset E$ we have
\(
\sup_{x\in \Gamma}t(x)<\infty.
\)
\end{assumption}

\begin{lemma}\label{boundedness}
Under assumption \ref{ass:positrec}
\begin{equation}\label{eq13}
c(x,z)\leq w_\alpha(x) \leq  \ee^x\left\{D_U\right\} \|f-\alpha v_\alpha(z)\|+\kappa.
\end{equation}
\end{lemma}
\begin{proof}
Define $w_\alpha^L (x) = v_\alpha^L(x) - v_\alpha^L (z)$. Similarly as above, we show
\[
w_\alpha^L(x) = \sup_{\tau} \ee^x\bigg\{ \int_0^\tau e^{-\alpha s} (f(X_s)-\alpha v_\alpha(z))ds + e^{-\alpha \tau} M_L w^L_\alpha(X_\tau) \bigg\}
\]
and $\sup_{x \in U} |w_\alpha^L (x)|  \le \kappa$. Since $M_L w^L_\alpha$ is bounded, standard supermartingale results yield that for any stopping time $\sigma$
\begin{multline*}
w^L_{\alpha}(x)=\sup_\tau \ee^x\bigg\{ \int_0^{\tau \wedge \sigma} e^{-\alpha s} (f(X_s)-\alpha v^L_\alpha(z))ds
+ \ind{\tau<\sigma}e^{-\alpha \tau}M_L w^L_\alpha(X_\tau)
+ \ind{\sigma\leq \tau}e^{-\alpha \sigma} w^L_{\alpha}(X_{\sigma})\bigg\}.
\end{multline*}
Apply the above formula for $\sigma = D_U$ and, taking into account negativity of $c$ and the upper bound on $w_\alpha^L$ on $U$, observe that $M_L w^L_\alpha(X_\tau) \le \kappa$ and $w^L_\alpha(X_{D_U}) \le \kappa$. Hence
\(
w^L_\alpha(x) \leq  \ee^x\left\{D_U\right\} \|f-\alpha v^L_\alpha(z)\|+\kappa.
\)
Since by Theorem \ref{disc prob} $v_\alpha^L(x)$ converges to $v_\alpha(x)$ uniformly in $x$ from compact sets, taking in the above inequality the limit $L \to \infty$ gives \eqref{eq13}. Finally, by Lemma \ref{oszac}, $c(x,z)\leq w_\alpha(x)$.
\end{proof}

\begin{lemma}\label{lem:from_above} For each $x\in E$
\begin{equation}\label{eq15}
\liminf_{\alpha\to 0}\alpha v_\alpha(x)\geq \mu(f).
\end{equation}
\end{lemma}
\begin{proof}
Let $R_\alpha f(x):=\ee^x\left\{\int _0^\infty e^{-\alpha s} f(X_s)ds\right\}$ be the resolvent of $f$. From $v_\alpha(x)\geq R_\alpha f(x)$ we have
\[
\liminf_{\alpha \to 0} \alpha v_\alpha(x)\geq \liminf_{\alpha \to 0} \alpha R_\alpha f(x)=\liminf_{\alpha \to 0}\int_0^\infty e^{-u}P_{\frac{u}{\alpha}}f(x)du = \mu(f).
\]
\end{proof}

Recall that $q_\alpha(x)=\ee^x\left\{\int _0^\infty e^{-\alpha s} \big(f(X_s) - \mu(f)\big) ds\right\}$. We shall assume that $q_\alpha$ is uniformly in $\alpha$ bounded from below.
\begin{assumption}
\item[(B4)]\label{ass:boundedpotential}
\(\displaystyle \sup_{\alpha \in [0,1]}\|q_\alpha^-\|<\infty,
\)
where $q_\alpha^-$ stands for the negative part of $q_\alpha$.
\end{assumption}

\begin{lemma}\label{lem:bounded from below}
Under \ref{ass:weak_feller}, \ref{ass:speed_ergodic} and \ref{ass:positrec}, if the  set $K_f:=\left\{x \in E:\ f(x)\leq \mu(f)\right\}$  is compact then \ref{ass:boundedpotential} holds.
\end{lemma}
\begin{proof}
We have $f(x) \ge \mu(f)$ on $K_f^c$, so
\begin{align*}
q_\alpha(x)
&=\ee^x\left\{\int_0^{D_U\wedge D_{K_f}}e^{-\alpha s}(f(X_s)-\mu(f))ds + e^{-\alpha D_U\wedge D_{K_f}} q_\alpha(X_{D_U\wedge D_{K_f}})\right\} \\
&\ge \ee^x\left\{e^{-\alpha D_U\wedge D_{K_f}} q_\alpha(X_{D_U\wedge D_{K_f}})\right\}\geq  -\sup_{y\in U\cup K_f}q_\alpha^-(y),
\end{align*}
where $D_{K_f} = \inf\{t \ge 0:\ X_t \in K_f\}$. By Lemma \ref{lq}, $q_\alpha$ is continuous, and converges to $q$, as $\alpha\to 0$, uniformly on compact sets, hence the last term in the expression above is uniformly bounded in $\alpha \in [0,1]$.
\end{proof}

Since for any stopping time $\tau \in \es$ and $\alpha > 0$
\[
q_\alpha(x)=\ee^x\left\{\int_0^\tau e^{-\alpha s} (f(X_s)-\mu(f))ds + e^{-\alpha \tau}q_\alpha(X_\tau)\right\}
\]
we obtain from \eqref{eq12}
\begin{equation}\label{eq16}
w_\alpha(x)-q_\alpha(x)=\sup_{\tau \in\es} \ee^x\left\{\int_0^{\tau} e^{-\alpha s} \big(\mu(f)-\alpha v_\alpha(z)\big)ds + e^{-\alpha \tau}\big(Mw_\alpha(X_{\tau})-q_\alpha(X_{\tau})\big)\right\}.
\end{equation}
Clearly, $\ve$-optimal stopping times in \eqref{eq12} and \eqref{eq16} coincide. In the following lemma we provide an upper bound on them.
\begin{lemma}\label{stopbound}
Assume \ref{ass:boundedpotential} and that $v:=\limsup_{\alpha \to 0}\alpha v_\alpha(z)>\mu(f)$. Then for any $\delta<v-\mu(f)$ and any  $\alpha\in \Lambda:=\left\{\alpha': \alpha' v_{\alpha'}(z)>\mu(f)+\delta\right\}$ we may restrict ourselves in \eqref{eq12} and \eqref{eq16} to stopping times $\tau$ satisfying the bound
\begin{equation}\label{eq17-}
\ee^x\left\{\frac1{\alpha}(1-e^{-\alpha \tau})\right\}-\frac1{\alpha v_{\alpha}(z)-\mu(f))}\ee^x\left\{e^{-\alpha \tau}\bar{c}(X_\tau)\right\}\leq Z(x),
\end{equation}
where $Z(x) = \sup_{\alpha \in (0,1)} \frac{\kappa + \epsilon+ \|q_{\alpha}^-\| + q_{\alpha}(x)-c(x,z)}{\delta}$ for an arbitrarily small $\ve > 0$. Moreover, $Z(x)$ is bounded on compact sets.
 \end{lemma}
\begin{proof}
Lemma \ref{lq} and assumption \ref{ass:boundedpotential} imply that $\sup_{\alpha \in (0,1)} \big(\|q_{\alpha}^-\| + q_{\alpha}(x)-c(x,z)\big)$ is bounded on compact sets and, therefore, so is $Z(x)$. For a given $\ve>0$, every $\ve$-optimal stopping time in \eqref{eq16} satisfies
\begin{equation*}%\label{eq18}
w_\alpha(x)-q_\alpha(x)-\ve\leq (\mu(f)-\alpha v_\alpha(z))\ee^x\left\{\frac1\alpha(1-e^{-\alpha \tau})\right\}+ \ee^x\left\{e^{-\alpha \tau}\sup_{a\in U}c(X_\tau,a)\right\}+\kappa + \|q_\alpha^-\|.
\end{equation*}
Therefore,
\begin{multline*}%\label{eq19}
(\alpha v_\alpha(z)-\mu(f))\ee^x\left\{\frac1{\alpha}(1-e^{-\alpha \tau})\right\}-\ee^x\left\{e^{-\alpha \tau}\bar{c}(X_\tau)\right\}\\
\le
\kappa + \ve+ \|q_\alpha^-\| + q_\alpha(x)-c(x,z) \le Z(x)(\alpha v_\alpha(z)-\mu(f)),
\end{multline*}
from which we obtain \eqref{eq17-}. 
\end{proof}

Complementing the above result are the following simple lemmas.
\begin{lemma}\label{lem:expon_tchebyshev}
%The mappting $\alpha \mapsto \frac{1}{\alpha} (1-e^{-\alpha t})$ is decreasing for $t \ge 0$ and $\alpha > 0$ and $\lim_{\alpha \to 0} \frac{1}{\alpha} (1-e^{-\alpha t}) = t$.
For any non-negative random variable $\tau$ and $\alpha > 0$
\[
\prob\{\tau > T\} \le \frac{\ee \big\{\frac{1}{\alpha} (1-e^{-\alpha \tau})\big\}}{\frac{1}{\alpha} (1-e^{-\alpha T})}.
\]
\end{lemma}
\begin{proof}
%The first part requires a direct verification. For the second statement 
Notice that $t \mapsto \frac{1}{\alpha} (1-e^{-\alpha t})$ is increasing for $\alpha > 0$, hence
\[
\ee \big\{\frac{1}{\alpha} (1-e^{-\alpha \tau})\big\} \ge \prob \{\tau > T\} \frac{1}{\alpha} (1-e^{-\alpha T}).
\]
\end{proof}

\begin{lemma}\label{uniform}
The mapping $x\mapsto Mw_\alpha(x)$ is uniformly in $\alpha \in (0,1)$ equicontinuous on each compact subset of $E$.
\end{lemma}

\begin{proof}
The assertion is a consequence of the estimate
\(
|Mw_\alpha(x)-Mw_\alpha(x')|\leq \sup_{\xi \in U}|c(x,\xi)-c(x',\xi)|.
\)
\end{proof}

Recalling that $\underline{c}(x)=\inf_{a\in U}c(x,a)$, we assume
\begin{assumption}
\item[(B5)]\label{ass:unifintc_b} For any compact set $\Gamma\subset E$ and a sequence of events $A_T \in \ef_T$, $T > 0$, we have
\[
\lim_{T\to \infty}\sup_{x\in \Gamma}\prob^x\left\{A_T\right\}= 0 \quad\implies\quad \lim_{T\to \infty}\sup_{x\in \Gamma}\ee^x \{\ind{A_T} |\underline{c}(X_T)| \}= 0.
\]
\end{assumption}

In a classical case when $\underline{c}$ is bounded, \ref{ass:unifintc_b} is trivially satisfied.

\begin{theorem}\label{thm:bellman}
Under \ref{ass:weak_feller}-\ref{ass:speed_ergodic}, \ref{ass:cost_function}-\ref{ass:unifintc_b}, if  $\limsup_{\alpha \to 0}\alpha v_\alpha(z)=:\lambda>\mu(f)$ then there exist a continuous function $w$ which is a solution to the following equation
\begin{equation}\label{eq22}
w(x)=\sup_\tau \liminf_{T \to \infty} \ee^x\left\{\int_0^{\tau \wedge T} (f(X_s)-\lambda)ds + Mw(X_{\tau \wedge T})\right\}.
\end{equation}
Moreover $w(z)=0$,
\begin{equation}\label{eq22'}
c(x,z)\leq w(x) \leq  \ee^x\left\{D_U\right\} \|f-\lambda\|+\kappa,
\end{equation}
and
\begin{equation}\label{eq22''}
\underline{c}(x)-\kappa\leq Mw(x)\leq \kappa.
\end{equation}
For any impulse strategy $V=(\tau_i, \xi_i)$, such that $\ee^x\left\{\tau_i\right\}<\infty$ for each $i$, we have that
\begin{equation}\label{eq230}
w(x)\geq \ee^x \Big\{ \int_0^{\tau_n}  \big(f(X_s)-\lambda\big) ds + \sum_{i=1}^n c(X_{\tau_i-}, \xi_i) + w(\xi_n) \Big\},
\end{equation}
where $(X_s)$ denotes the process controlled by the strategy $V$. We have equality in \eqref{eq230} for the strategy $V^*=(\tau_i^*,\xi_i^*)$ defined as follows:
$\tau_1^*=\inf\left\{s\geq 0: w(X_s)=Mw(X_s)\right\}$, $\tau_{n+1}^*=\tau_n^*+\tau_1^* \circ \theta_{\tau_n^*}$, where $\theta_t$ is a Markov shift operator, and $\xi_n^*=\hat{\xi}(X^n_{\tau_n^*})$, where $\hat{\xi}:E \mapsto U$ is a Borel measurable function such that $Mw(y)=c(y,\hat{\xi}(y))+w(\hat{\xi}(y))$ for $y\in E$. Moreover, $x \mapsto \ee^x\{\tau^*_1\}$ is bounded on compact sets.
\end{theorem}

\begin{proof}
By local compactness of the state space $E$ and Lemma \ref{uniform} there is a continuous function $v$ such that $Mw_{\alpha}(x)\to v(x)$ uniformly on compact sets over a suitable sequence of $\alpha \to 0$. Therefore, we can choose a sequence $\alpha_n \to 0$ such that $\lim_{n\to \infty} \alpha_n v_{\alpha_n}(z)=\lambda$, $\alpha_n v_{\alpha_n}(z) > \mu(f) + \delta$ for some $\delta > 0$, and $Mw_{\alpha_n}(x)\to v(x)$ uniformly on compact sets.
Let
\begin{equation}\label{eq23}
w(x):= \sup_\tau \liminf_{T \to \infty} \ee^x\bigg\{\int_0^{\tau \wedge T} \big(f(X_s)-\lambda\big)ds + v(X_{\tau \wedge T})\bigg\}.
\end{equation}
We are going to show that along a subsequence $w_{\alpha_n}(x)\to w(x)$ uniformly on compact subsets as $n\to  \infty$.
For this purpose we consider finite time approximations. Let
\begin{equation}\label{eq24}
w_{\alpha_n}^T(x)=\sup_\tau \ee^x\bigg\{\int_0^{\tau \wedge T} e^{-{\alpha_n} s} \big(f(X_s)-{\alpha_n} v_{\alpha_n}(z)\big)ds + e^{-{\alpha_n} \tau \wedge T}Mw_{\alpha_n}(X_{\tau \wedge T})\bigg\},
\end{equation}
and
\(
w^T(x) = \sup_\tau \ee^x\Big\{\int_0^{\tau\wedge T}  (f(x_s)-\lambda)ds + v(x_{\tau\wedge T})\Big\}.
\)
Then
\begin{equation}\label{eq24+}
w(x) - w_{\alpha_n}(x) = \big(w(x) - w^T(x)\big) + \big(w^T(x) - w^T_{\alpha_n}(x)\big) + \big(w^T_{\alpha_n}(x) - w_{\alpha_n}(x)\big) = (I) + (II) + (III).
\end{equation}
To address the convergence of the third term of \eqref{eq24+} we write
\begin{equation}\label{eq25}
\begin{aligned}
0
&\le w_{\alpha_n}(x)-w_{\alpha_n}^T(x)\\
&\le \sup_{\tau \in \es} \ee^x\bigg\{\int_0^\tau e^{-\alpha_n s} (\mu(f)-\alpha_n v_{\alpha_n}(z))ds + e^{-\alpha_n \tau}(Mw_{\alpha_n}(X_\tau) -q_{\alpha_n}(X_\tau)) \\
&\hspace{11pt} -\int_0^{\tau \wedge T} e^{-\alpha_n s} (\mu(f)-\alpha_n v_{\alpha_n}(z))ds  -  e^{-\alpha_n \tau \wedge T}(Mw_{\alpha_n}(X_{\tau\wedge T})-q_{\alpha_n}(X_{\tau \wedge T}))\bigg\}\\
&\le \sup_{\tau \in \es} \ee^x\bigg\{\int^\tau_{\tau\wedge T} e^{-\alpha_n s} (\mu(f)-\alpha_n v_{\alpha_n}(z))ds\\
&\hspace{11pt}+ \ind{\tau\geq T}\Big[e^{-\alpha_n \tau}Mw_{\alpha_n}(X_\tau)- e^{-\alpha_n T}Mw_{\alpha_n}(X_T)- e^{-\alpha_n \tau} q_{\alpha_n}(X_\tau)+ e^{-\alpha_n T}q_{\alpha_n}(X_T)\Big]\bigg\}.
\end{aligned}
\end{equation}
Recall from Lemma \ref{stopbound} that in the above we can restrict attention to stopping times which satisfy the bound
\begin{equation}\label{eqn25+}
\ee^x\left\{\frac1{\alpha_n}(1-e^{-\alpha_n \tau})\right\}\leq Z(x)
\end{equation}
for a function $Z(x)$ which is independent from $n$ and bounded on compact sets. 
%We can also require that these stopping times are hitting times of a closed set which is independent of $x$ (but depends on $\alpha_n$ because $\ve$-optimal stopping times for $w_{\alpha_n}(x)$ are of this form). This implies that events $A_T = \{\tau > T\}$ do not depend on $x$, which is required in assumption \ref{ass:unifintc_b}. 
Note also that for $\alpha>0$ we have
\begin{equation}\label{eq26}
\underline{c}(x)-\kappa\leq Mw_\alpha(x)\leq \kappa.
\end{equation}
Hence,
\begin{equation}\label{eq27}
\begin{aligned}
&\ee^x\bigg\{\int^\tau_{\tau\wedge T} e^{-\alpha_n s}\big(\mu(f)-\alpha_n v_{\alpha_n}(z)\big)ds\\
&\hspace{11pt}+\ind{\tau\geq T}\Big[e^{-\alpha_n \tau}Mw_{\alpha_n}(X_\tau)-e^{-\alpha_n T}Mw_{\alpha_n}(X_T)-e^{-\alpha_n \tau}q_{\alpha_n}(X_\tau)+e^{-\alpha_n T}q_{\alpha_n}(X_T)\Big]\bigg\}\\
&\le
\ee^x\left\{\ind{\tau\geq T}e^{-\alpha_n T}\big(2 \kappa + \|q_{\alpha_n}^-\|-\underline{c}(X_T)+q_{\alpha_n}(X_T)\big)\right\}\\
&\le (2 \kappa +\|q_{\alpha_n}^-\|) \frac{Z(x)}{\frac{1}{\alpha_n} (e^{\alpha_n T}-1)} + \ee^x\left\{\ind{\tau\geq T}e^{-\alpha_nT}\big(-\underline{c}(X_T)+q_{\alpha_n}(X_T)\big)\right\}\\
&\le (2 \kappa +\|q_{\alpha_n}^-\|) \frac{Z(x)}{T} + \ee^x\left\{\ind{\tau\geq T}e^{-\alpha_nT}\big(-\underline{c}(X_T)+q_{\alpha_n}(X_T)\big)\right\},
\end{aligned}
\end{equation}
where we used \eqref{eqn25+} and Lemma \ref{lem:expon_tchebyshev} and finally the fact that $e^{\alpha_nT}-1\geq \alpha_n T$.
% From the first assertion of this lemma, there is a subsequence $\alpha_{n_k}$ and an increasing sequence $T_k$ converging to $\infty$ %such that $\frac{1}{\alpha_{n_k}} (1-e^{-\alpha_{n_k} T_k}) \ge T_k /2$. For a sufficiently small $\ve > 0$
%\begin{equation}\label{eq27}
%w_{\alpha_{n_k}} (x) - w_{\alpha_{n_k}}^{T_k} (x) \le \ve + \Big(2 \kappa +\sup_{a \in (0,1)} \|q_{a}^-\|\Big) \frac{Z(x)}{T_k / 2} + %\ee^x\Big\{\ind{\tau_k\geq T}\big(-\underline{c}(X_T)+\sup_{a \in (0,1)}q_{a}(X_T)\big)\Big\}
%\end{equation}
%where $\tau_k$ is $\ve$-optimal stopping time for $w_{\alpha_{n_k}}$ of a Markovian type, i.e., as mentioned before, a first hitting %time of a closed set.
Therefore, by assumptions \ref{ass:boundedpotential}-\ref{ass:unifintc_b}, \eqref{eqn:eq0potent_AT} and \eqref{eq27}, for any $\eta > 0$ and any compact set $\Gamma$ there is $T$ such that $w_{\alpha_{n}}(x) - w_{\alpha_{n}}^{T}(x) \le \eta$ for all $x \in \Gamma$ and all $n$.

%Lemma \ref{lem:bounded from below} yields the convergence of the first term of \eqref{eq27} to $0$ as $k \to \infty$ and this %convergence is uniform on compact subsets of $E$ due to boundedness of $Z(x)$ on compacts. Assumption \ref{ass:unifintc} yields uniform %convergence of $\ee^x\{\ind{\tau_k > T_k} \underline{c}(X_{T_k}) \}$ on compact sets. The convergence of $\ee^x\{\ind{\tau_k > T_k} %q_{\alpha_{n_k}}(X_{T_k}) \}$ is implied by Lemma \ref{lq}. Summarising, for any $\eta > 0$ and any compact set $\Gamma$ there is $T^*$ %and $\alpha^*$ such that $w_{\alpha_{n_k}}(x) - w_{\alpha_{n_k}}^{T_k}(x) \le \eta$ for all $x \in \Gamma$, $T_k \ge T^*$ and %${\alpha_{n_k}} \ge \alpha^*$.

From \eqref{eq26} we have that
\begin{equation}\label{eq28}
\underline{c}(x)-\kappa\leq v(x)\leq \kappa.
\end{equation}
Notice that 
\begin{equation}\label{eq30}
\begin{aligned}
|\ee^x\left\{Mw_{\alpha_n}(X_{\tau\wedge T})-v(X_{\tau\wedge T})\right\}|
& \leq \left|\ee^x\left\{\ind{\rho(x,X_{\tau\wedge T})\leq R}(Mw_{\alpha_n}(X_{\tau\wedge T})-v(X_{\tau\wedge T}))\right\}\right|\\
&\hspace{11pt} + \ee^x\left\{\ind{\rho(x,X_{\tau\wedge T})\geq R}(2\kappa-\underline{c}(X_{\tau\wedge T}))\right\}=a_{\alpha_n} + b_R.
\end{aligned}
\end{equation}
For a fixed $R$ we have that $\lim_{n \to \infty} a_{\alpha_n} = 0$ for $x$ in compact sets by the definition of $v$ in the beginning of the proof. The term $b_R$ can be made arbitrarily small uniformly in $x$ in compact subsets of $E$, since
\(
b_R \le \ee^x\left\{\ind{\rho(x,X_{\tau\wedge T})\geq R}(2\kappa + L)\right\} + \ee^x \{ \ind{\zeta_T > L} \zeta_T\},
\)
where $\zeta_T$ is defined in \ref{ass:unifintcalpha}. Now letting $R \to \infty$ (using assumption \ref{ass:weak_feller} and \cite[Proposition 2.1]{Palczewski2008}) and then $L \to \infty$ we obtain that $b_R\to  0$. 
Hence, for each fixed $T$ we have uniformly in $x$ in compact subsets of $E$ that
\begin{equation}\label{eq29}
\begin{aligned}
w^T_{\alpha_n}(x) = \sup_\tau \ee^x\bigg\{\int_0^{\tau\wedge T} e^{-\alpha_n s}(f(X_s)-\alpha_n v_{\alpha_n}(z))ds + e^{-\alpha_n {(\tau\wedge T)}} Mw_{\alpha_n}(X_{\tau\wedge T})\bigg\}\\
 \to \sup_\tau \ee^x\bigg\{\int_0^{\tau\wedge T}  (f(X_s)-\lambda)ds + v(X_{\tau\wedge T})\bigg\}=w^T(x),
\end{aligned}
\end{equation}
which provides a uniform on compacts bound on term (II) of \eqref{eq24+}.

Finally, we estimate term (I) of \eqref{eq24+}. From the form \eqref{eq23} of $w(x)$ by Lemma \ref{lq} using \eqref{eq0potent} we obtain
\begin{equation}\label{eq31}
w(x)-q(x)=\sup_{\tau \in \bes} \liminf_{T \to \infty} \ee^x\left\{(\mu(f)-\lambda)(\tau \wedge T)+ v(X_{\tau \wedge T})-q(X_{\tau \wedge T})\right\}.
\end{equation}
Since for each $\ve> 0$ there is a bounded $\ve$-optimal stopping time $\tau$, in analogy to the proof of Lemma \ref{stopbound}, using \eqref{eq28}, we obtain
\[
\underline{c} (x) - \kappa - q(x) - \ve 
 \le v(x) - q(x) - \ve \le w(x) - q(x) - \ve
\le (\mu(f)-\lambda) \ee^x\{\tau\}+ \kappa + \|q^-\|.
\]
Therefore, we may restrict ourselves in \eqref{eq31} as well as in \eqref{eq23} to stopping times satisfying
\begin{equation}\label{eq32}
\ee^x\left\{\tau\right\} \leq \frac{2\kappa + \|q^-\|+q(x)-\underline{c}(x) + 1}{\lambda - \mu(f)}.
\end{equation}
Consequently, similarly to \eqref{eq25} we have
\begin{equation}\label{eq33}
\begin{aligned}
0\le w(x)-w^T(x)
&\le \sup_\tau\left\{\ind{\tau\geq T}\big(v(X_\tau)+\|q^-\|-v(X_T)+q(X_T)\big)\right\}\\
&\le \sup_\tau\left\{\ind{\tau\geq T}\big(2\kappa+\|q^-\|-\underline{c}(X_T)+q(X_T)\big)\right\}.
\end{aligned}
\end{equation}
Since we may restrict ourselves to stopping times $\tau$ satisfying \eqref{eq32}, Tchebyshev inequality, Lemma \ref{lq} and assumptions \ref{ass:boundedpotential}-\ref{ass:unifintc_b} imply that $w^T(x)\to w(x)$ uniformly in $x$ from compact subsets of $E$.

Summarizing now \eqref{eq27}, \eqref{eq29} and \eqref{eq33} we obtain that $w_{\alpha_{n}}(x)\to w(x)$  uniformly in $x$ from compact subsets of $E$.
Consequently, $Mw_{\alpha_{n}}(x)\to Mw(x)$ uniformly in $x$ from compact subsets of $E$. This proves that $v(x)=Mw(x)$ which completes the proof of the first part of Theorem. Notice that \eqref{eq22'} follows directly from \eqref{eq13}, while \eqref{eq22''} follows from \eqref{eq26}.

Take any impulse strategy $V=(\tau_i,\xi_i)$ with integrable impulse times. For any $\alpha>0$, by strong Markov property of $(X_t)$ and using approximations with bounded cost operators $M_L$ as in the proof of Theorem \ref{disc prob} we show
\[
v_\alpha (x) \ge \ee^x \Big\{ \int_0^{\tau_k}e^{-\alpha s} f(X_s) ds + \sum_{i=1}^k e^{-\alpha \tau_i}c(X_{\tau_i-}, \xi_i) + e^{-\alpha \tau_k} v_\alpha(\xi_k) \Big\}.
\]
Subtract $v_\alpha(z)$ from both sides to get
% By the proof of Theorem \ref{disc prob}, $\te^n 0(x)$ is the value of the functional $J_\alpha$ with at most $n$ impulses and its proof
 \begin{equation}\label{eq231}
 w_\alpha(x)\geq \ee^x \Big\{ \int_0^{\tau_k}e^{-\alpha s} \left(f(X_s)-\alpha v_\alpha(z)\right) ds + \sum_{i=1}^k e^{-\alpha \tau_i}c(X_{\tau_i-}, \xi_i) + e^{-\alpha \tau_k} w_\alpha(\xi_k) \Big\}.
 \end{equation}
Since $w_{\alpha_{n}}$ converges uniformly on compact sets to $w$, $\lim_{n \to \infty} \alpha_{n} v_{\alpha_{n}}(z) = \lambda$, and $\ee^x\left\{\tau_k\right\}<\infty$  we obtain \eqref{eq230} from \eqref{eq231}.
By \cite[Theorem 4.8]{Palczewski2016a}\footnote{All assumptions of Theorem 4.8 in \cite{Palczewski2016a} apart from (C3) are trivially satisfied. Assumption (C3) follows from \ref{ass:boundedpotential} and \cite[Remark 4.6]{Palczewski2016a}.} the stopping time $\tau_1^*$ is optimal for the Bellman equation \eqref{eq22}. By \eqref{eq32} we have that $x\mapsto \ee^x\left\{\tau_1^*\right\}$ is bounded on compact sets. Therefore for strategy $V^*$ we have equality in \eqref{eq230}, which completes the proof.
\end{proof}

\begin{prop}\label{propseparated}
Under assumptions of Theorem \ref{thm:bellman} if the cost for impulses is in a separated form $c(x,\xi)=d(x) + e(\xi)$, where $d$ and $e$ are continuous functions, we have
\begin{equation}\label{equiv1}
\sup_{x\in U}\sup_\tau \liminf_{T \to \infty} \ee^x\bigg\{\int_0^{\tau \wedge T} (f(X_s)-\lambda)ds + d(X_{\tau \wedge T}) +e(x)\bigg\}=0,
\end{equation}
and
\begin{equation}\label{equiv2}
\lambda=\sup_{x\in U} \sup_\tau \frac{\ee^x\left\{\int_0^\tau f(X_s)ds + d(X_\tau) +e(x)\right\}} {\ee^x\left\{\tau\right\}}.
\end{equation}
The suprema in \eqref{equiv1} and \eqref{equiv2} are attained for $\hat{x}=argmax_{\xi\in U}\left[w(\xi) + e(\xi)\right]$ and \(\hat{\tau}=\inf\big\{s\geq 0: w(X_s)=Mw(X_s)\big\}\).
Furthermore, the measure
\begin{equation}\label{invmeas}
\eta(A):=\frac{\ee^{\hat{x}}\left\{\int_0^{\hat{\tau}} \ind{A}(X_s)ds \right\}}{\ee^{\hat{x}}\left\{\hat{\tau}\right\}}.
\end{equation}
for $A\in {\cal E}$ is a unique invariant measure for controlled process $(X_s^*)$ using the strategy $V^*=(\tau_i^*,\xi_i^*)$ defined as
$\tau_1^*= \hat \tau$, $\tau_{n+1}^*=\tau_n^*+\hat \tau^* \circ \theta_{\tau_n^*}$, and $\xi_i^* = \hat x$.
%in Theorem \ref{thm:bellman}.
\end{prop}
\begin{proof}
Note that $Mw(x)=\sup_{\xi\in U}\left[w(\xi)+e(\xi)\right]+d(x)$. Then \eqref{eq22} has the form
\[
w(x)=\sup_\tau \liminf_{T \to \infty} \ee^x\bigg\{\int_0^{\tau \wedge T} (f(X_s)-\lambda)ds + d(X_{\tau \wedge T}) + \sup_{\xi\in U}\left[w(\xi)+e(\xi)\right]\bigg\}.
\]
When $x\in U$ we have
\[w(x)\geq \sup_\tau \liminf_{T \to \infty} \ee^x\bigg\{\int_0^{\tau \wedge T} (f(X_s)-\lambda)ds + d(X_{\tau \wedge T}) + \left[w(x)+e(x)\right]\bigg\},\] from which
\begin{equation}\label{eqequiv12}
0\geq \sup_\tau \liminf_{T \to \infty} \ee^x\bigg\{\int_0^{\tau \wedge T} (f(X_s)-\lambda)ds + d(X_{\tau \wedge T}) + e(x)\bigg\}
\end{equation}
with equality for $x=\hat{x}$. Recall from Theorem \ref{thm:bellman} that $\ee^x \{ \hat\tau \} < \infty$ for $x \in E$, so in \eqref{eqequiv12} we can consider integrable stopping times only and, therefore, skip the limit (c.f. \cite[Lemma 4.2]{Palczewski2016a}). Hence, for any stopping time $\tau$ and $x\in U$ such that $\ee^x\left\{\tau\right\}<\infty$ we obtain
\[
\lambda\geq  \frac{\ee^x\big\{\int_0^\tau f(X_s)ds + d(X_\tau) +e(x)\big\}}{\ee^x\{\tau\}}
\]
with equality whenever $x=\hat{x}$ and $\tau=\hat{\tau}$.
Finally, under control $V^*=(\tau_i^*,\xi_i^*)$ the controlled process $(X_s^*)$ is Markovian with the transition operator 
\(%\label{transoper}
\prob^*(x,A):= \allowbreak \sum_{i=0}^\infty \ee^x\big\{\ind{\tau_i^*\leq t < \tau_{i+1}^*} \ind{A}(x_t^{i+1})\big\},
\)
with $\tau_0^*:=0$ and $x_t^i$ as defined in Section \ref{sec:intro} where the construction of controlled process was sketched. By direct calculation, similarly to the formula (4.14) in the proof of \cite[Theorem 4.1]{Khasminskii2012}, one can show that $\eta$ defined in \eqref{invmeas} is in fact an invariant measure for $(X_s^*)$. Since $\ee^x\left\{\tau_1^*\right\}<\infty$ for $x\in E$,  the process $(X_s^*)$ enters $\hat{x}$ infinitely often and therefore $\eta$ is the unique invariant measure.
\end{proof}

The proof of the following lemma follows immediately from \cite[Lemma 4.11]{Palczewski2016a}.
\begin{lemma}\label{lem:super_mart}
Under assumptions of Theorem \ref{thm:bellman} , the process $Z_t := \int_0^t \big(f(X_s)- \lambda\big) ds + w(X_t)$ is a right-continuous $\prob^x$-supermartingale for any $x \in E$. Moreover, for a bounded stopping time $\sigma$ and an arbitrary stopping time $\tau$
\begin{equation}\label{eqn:undisc_bellman_ineq}
\ee^{x} \Big \{ \int_0^\sigma \big(f(X_s)- \lambda\big) ds + Mw(X_\sigma) \Big\}
\le \ee^{x} \Big\{ \int_0^{\tau \wedge \sigma} \big(f(X_s)- \lambda\big) ds + \ind{\sigma < \tau} Mw(X_{\sigma}) + \ind{\sigma \ge \tau} w(X_{\tau}) \Big\}.
\end{equation}
\end{lemma}

\begin{lemma}
Under assumptions of Theorem \ref{thm:bellman}
\begin{equation}\label{wbounds}
\begin{aligned}
&\max\left\{c(x,z),\ee^x\left\{D_U\right\} (-\|f\|-\lambda) - \kappa -\|\underline{c}\|_U \right\}\\
&\le w(x) \le \min\Big\{q(x) + \kappa + \sup_{\alpha \in (0,1)} \|q_\alpha^-\|,  \ee^x\left\{D_U\right\} \|f-\lambda\|+\kappa\Big\}
\end{aligned}
\end{equation}
for $x\in E$, where $\|\underline{c}\|_U=\sup_{y\in U}|\underline{c}(y)|$. If, additionally, $f(x)-\lambda \leq 0$ for $x$ outside of some compact set $K$, then $w$ is bounded from above.
\end{lemma}

\begin{proof}
In view of \eqref{eq22'},  to prove \eqref{wbounds} it remains to show
\begin{equation}\label{eq33'}
\ee^x\left\{D_U\right\} (-\|f\|-\lambda) - \kappa -\|\underline{c}\|_U \le w(x) \le q(x) + \kappa + \sup_{\alpha \in (0,1)} \|q_\alpha^-\|.
 \end{equation}
%  and
%  \begin{equation}%\label{eq33''}
%   w(x)\leq q(x) + \kappa + \sup_{\alpha \in (0,1)} \|q_\alpha^-\|.
%  \end{equation}
From \eqref{eq22} we have that
\(
w(x)\geq \ee^x\left\{\int_0^{D_U}(f(X_s)-\lambda)ds + Mw(X_{D_u})\right\}
\)
and, therefore, by \eqref{eq22''} we obtain the first inequality in \eqref{eq33'}.
Combining \eqref{eq16} with \eqref{eq22''} yields
\begin{equation}\label{eq:34}
w_\alpha(x)-q_\alpha(x)\leq\sup_\tau \ee^x\bigg\{\int_0^\tau e^{-\alpha s} (\mu(f)-\alpha v_\alpha(z))ds + e^{-\alpha \tau}(\kappa+ \sup_{\alpha \in (0,1)} \|q_\alpha^-\|)\bigg\}.
\end{equation}
Take a sequence $\alpha_n\to 0$ such that $w_{\alpha_n}(x) \to w (x)$. By Lemma \ref{lq} $q_{\alpha_n}(x)\to q(x)$. Hence \eqref{eq:34} implies $w(x)\leq q(x)+\kappa+\sup_{\alpha \in (0,1)} \|q_\alpha^-\|)$, which completes the proof of the second inequality in \eqref{eq33'}.

Let now $\Gamma$ be a compact set that  contains the sets $U$ and $K$. Since the supremum in the definition of $w$ can be taken over bounded stopping times, from Lemma \ref{lem:super_mart} we get
\begin{equation}\label{eq:35}
w(x)\le\sup_{\tau \text{-bounded}} \ee^x\bigg\{\int_0^{\tau \wedge D_\Gamma}  (f(X_s)-\lambda)ds + \ind{\tau<D_\Gamma} Mw(X_\tau)+ \ind{D_\Gamma\leq \tau}w(X_{D_\Gamma})\bigg\}.
\end{equation}
Using \eqref{eq22''} and observing that the integrand is negative outside of $\Gamma$, we obtain
\[
w(x)\leq \sup_\tau \ee^x\left\{\ind{\tau<D_\Gamma} \kappa + \ind{D_\Gamma\leq \tau}\|w\|_\Gamma \right\},
\]
where $\|w\|_\Gamma=\sup_{y\in \Gamma} |w(y)|<\infty$ by the continuity of $w$. Consequently $w(x)\leq \kappa \vee \|w\|_\Gamma$, which completes the proof.
\end{proof}

To infer from the solution of the Bellman equation \eqref{eq22} that $\lambda$ is the optimal value, we will need the following Tauberian theorem.
\begin{lemma}\label{lem:tauber} For a bounded function $f$ and sequences of random variables $Y_i \le 0$, $\tau_i \ge 0$ with $(\tau_i)$ being an increasing sequence we have
\begin{equation}\label{eq:37}
\begin{aligned}
&\liminf_{T \to \infty} \frac{1}{T} \bigg(\int_0^T f(s) ds+ \ee\Big\{\sum_{i=1}^\infty \ind{\tau_i \le T} Y_i \Big\}\bigg)\\ 
&\leq
\liminf_{\alpha \to 0} \alpha \bigg( \int_0^\infty e^{-\alpha s} f(s) ds + \ee \Big\{ \sum_{i=1}^\infty \ind{\tau_i < \infty} e^{-\alpha \tau_i} Y_i \Big\}\bigg).
\end{aligned}
\end{equation}
\end{lemma}
\begin{proof}
Let $a = \liminf_{T \to \infty} \frac{1}{T} \left(\int_0^T f(s) ds+ \ee\Big\{\sum_{i=1}^\infty \ind{\tau_i \le T} Y_i \Big\}\right)$. If $a = -\infty$ then the inequality is obvious. Otherwise, for every $\ve > 0$ there is $M > 0$ such that
\[
a-\ve \leq  \frac{1}{T} \Big(\int_0^T f(s) ds + \ee\Big\{\sum_{i=1}^\infty \ind{\tau_i \le T} Y_i \Big\}\Big)
\]
for all $T \ge M$. Using the representation $e^{-\alpha t} = \int_t^\infty \alpha e^{-\alpha u} du$ we write
\begin{multline}\label{Taubereq1}
\alpha\bigg( \int_0^\infty e^{-\alpha s} f(s) ds + \ee \Big\{ \sum_{i=1}^\infty  \ind{\tau_i < \infty} e^{-\alpha \tau_i} Y_i \Big\}\bigg) \\
 = \int_0^\infty f(s) \int_s^\infty \alpha^2 e^{-\alpha u} du ds + \ee \Big\{ \sum_{i=1}^\infty Y_i \int_0^\infty \alpha^2 e^{-\alpha u} \ind{\tau_i\leq u}du \Big\}.
\end{multline}
For any $L>0$ and any positive integer $n$ we can apply Fubini's theorem:
\[
\ee\Big\{ \sum_{i=1}^n \big(Y_i\vee (-L)\big) \int_0^\infty \alpha^2 e^{-\alpha u} \ind{\tau_i\leq u}du \Big\}=
\int_0^\infty \alpha^2 e^{-\alpha u} \ee\Big\{\sum_{i=1}^n \ind{\tau_i\leq u} \big(Y_i\vee (-L)\big)\Big\}du.
\]
Letting $L\to \infty $ and $n\to \infty$, monotone convergence theorem yields
\[
\ee\Big\{ \sum_{i=1}^\infty Y_i \int_0^\infty \alpha^2 e^{-\alpha u} \ind{\tau_i\leq u}du \Big\}=
\int_0^\infty \alpha^2 e^{-\alpha u} \ee\Big\{\sum_{i=1}^\infty\ind{\tau_i\leq u} Y_i\Big\}du.
\]
Therefore from \eqref{Taubereq1} we obtain
\begin{align*}
&\alpha\bigg( \int_0^\infty e^{-\alpha s} f(s) ds + \ee \Big\{ \sum_{i=1}^\infty \ind{\tau_i < \infty} e^{-\alpha \tau_i} Y_i \Big\}\bigg)\\
&= \int_0^\infty \alpha^2 e^{-\alpha u} \bigg(\int_0^u f(s)ds + \ee \Big\{ \sum_{i=1}^\infty  Y_i \ind{\tau_i\leq u} \Big\}\bigg)du\\
&\ge \alpha^2 M \Big(\|f\| + \ee \Big\{ \sum_{i=1}^\infty  Y_i \ind{\tau_i\leq M} \Big\}\Big) +
(a-\ve)\int_M^\infty \alpha^2 e^{-\alpha u} u\, du.
\end{align*}
Letting $\alpha \to 0$, the right-hand side converges to $a-\ve$, since the finiteness of $a$ implies that \[ \ee \big\{ \sum_{i=1}^\infty \ind{\tau_i \le M}  Y_i \big\} > - \infty \]
for all $M > 0$. This completes the proof since $\ve$ can be taken arbitrarily small.
\end{proof}

Recall that a strategy $V=(\tau_i, \xi_i)$ is \emph{admissible} for $x \in E$ if stopping times $\tau_i$ increase to infinity $\prob^x$-a.s. If, further,  $\ee^x \{ \tau_i \} < \infty$ for all $i$, we call the strategy \emph{integrable}. The aim of the paper is to maximise two types of functionals: the functional $J(x, V)$ defined in \eqref{eqn:main_functional} over admissible strategies $V$ and the functional $\hat J(x, V)$ defined in \eqref{eqn:weaker_functional} over admissible integrable strategies $V$. The following theorem links the solution to the auxiliary Bellman equation \eqref{eq22} with the optimal value of the above functionals.
% For the convenience of the reader we recall two average cost per unit time functionals studied in this paper. A strategy $V=(\tau_i, \xi_i)$ is \emph{admissible} for $x \in E$ if stopping times $\tau_i$ increase to infinity $\prob^x$-a.s. If, further,  $\ee^x \{ \tau_i \} < \infty$ for all $i$, we call the strategy \emph{integrable}. Denoting by $(X_t)$ the controlled process, we wish to maximize over all admissible strategies the functional
% \[
% J\big(x, (\tau_i, \xi_i)\big) = \liminf_{T \to \infty} \frac{1}{T} \ee^x \Big\{ \int_0^T f(X_s) ds + \sum_{i = 1}^\infty \ind{\tau_i \le T} c(X_{\tau_i-}, \xi_i) \Big\},
% \]
% and over all integrable strategies the functional
% \[
% \hat{J}\big(x, (\tau_i, \xi_i)\big) = \liminf_{n \to \infty} \frac{1}{\ee^x\{\tau_n\}} \ee^x \Big\{ \int_0^{\tau_n} f(X_s) ds + \sum_{i = 1}^n \ind{\tau_i \le T} c(X_{\tau_i-}, \xi_i) \Big\},
% \]
% where $X_{\tau_i-}$ is the state of the process before the $i$-th impulse with an obvious meaning if there is more than one impulse at the same time (this will, however, never be optimal due to the assumptions on the cost function).

\begin{theorem}\label{optimalityoflambda}
Make the same assumptions as in Theorem \ref{thm:bellman}. Denote by $(\tau^*_i, \xi^*_i)$ the optimal strategy from Theorem \ref{thm:bellman}.
\begin{enumerate}
\item $\lambda = \sup_V \hat{J}\big(x, (\tau_i, \xi_i)\big)$ with the supremum over all integrable strategies $V=(\tau_i, \xi_i)$. The strategy $(\tau^*_i, \xi^*_i)$ realizes the supremum: $\lambda = \hat{J}\big(x, (\tau^*_i, \xi^*_i)\big)$.
\item$\lambda \ge J\big(x, (\tau_i, \xi_i)\big)$ for any admissible strategy $(\tau_i, \xi_i)$.
\item The strategy $(\tau^*_i, \xi^*_i)$ is optimal for the functional $J$, that is $\lambda = J\big(x, (\tau^*_i, \xi^*_i)\big)$, when $c(x,\xi)=e(\xi)$ (a separated cost with $d\equiv 0$), or when $w$ is bounded from above.
\end{enumerate}
\end{theorem}
\begin{proof}
From \eqref{eq230} for any integrable strategy $(\tau_i, \xi_i)$ we have
\begin{equation}\label{230cont}
w(x)+\lambda \ee^x\left\{\tau_n\right\}\geq \ee^x \Big\{ \int_0^{\tau_n}  f(X_s)ds + \sum_{i=1}^n c(X_{\tau_i-}, \xi_i) + w(\xi_n) \Big\}.
\end{equation}
Since $w$ is bounded on $U$, $\ee^x\left\{\tau_n\right\}<\infty$ and $\tau_n \to \infty$ we obtain
that $\lambda\geq \hat{J}\big(x,(\tau_i, \xi_i)\big)$ with equality for the strategy $(\tau^*_i, \xi^*_i)$ defined in Theorem \ref{thm:bellman}, which completes the proof of assertion 1.

Fix $x \in E$ and an admissible strategy $(\tau_i, \xi_i)$. Denote by $(X_s)$ the controlled process. Recalling that $v_\alpha$ is the discounted value function and $w_\alpha(x) = v_\alpha(x) - v_\alpha(z)$ we have
\[
w_{\alpha}(x) + v_{\alpha}(z) \ge \ee^x \Big\{ \int_0^\infty e^{-\alpha s} f(X_s) ds + \sum_{i=1}^\infty \ind{\tau_i < \infty} e^{-\alpha \tau_i} c(X_{\tau_i-}, \xi_i) \Big\}.
\]
Multiply both sides by $\alpha$ and take $\liminf_{\alpha \to 0}$ using Lemma \ref{lem:tauber} to show $\lambda \ge J\big(x, (\tau_i, \xi_i)\big)$. Here we also use the fact that $\liminf_{\alpha \to 0} w_\alpha(x) \le w(x)$.

In the case of separated cost we use Proposition \ref{propseparated} by which the measure $\eta$ defined in \eqref{invmeas} is invariant for $X^*$ controlled by the strategy $(\tau^*_i, \xi^*_i)$. Then for any $T>0$ and $\gamma$ we have that
\(\ee^\eta \big\{ \int_0^T (f(X^*_s)-\gamma) ds\big\}=T \eta(f-\gamma),\)
where $\ee^\eta$ means that the process starts with measure $\eta$. Moreover $X^*$ is a Harris Markov process. By ergodic theorem for Harris Markov processes (see Theorem II.1 of \cite{Azema1967}) we obtain that
\begin{equation}\label{eqlimit1}
\lim_{T\to \infty}\frac{1}{T} \ee^x \Big\{ \int_0^T (f(X^*_s)-\gamma) ds\Big\}= \eta(f-\gamma)
\end{equation}
for $\eta$ almost all $x$. To show that the above limit holds for all $x \in E$ use Assumption \ref{ass:positrec} which implies that $\ee^x\left\{D_U\right\}<\infty$ for any $x$.  Then by the proof of Theorem \ref{thm:bellman} (c.f. \cite[Theorem 4.8]{Palczewski2016a}) we have $\sup_{x\in U}\ee^x\left\{\hat{\tau}\right\}<\infty$, i.e., $\ee^x \{ \tau^*_1 \} < \infty$ for any $x \in E$, which implies \eqref{eqlimit1} for all $x$. Hence, in particular, for $\hat x$:
\begin{equation}\label{eqlimit}
\lim_{T\to \infty}\frac1T \ee^{\hat{x}} \Big\{ \int_0^T (f(X^*_s)-\gamma) ds\Big\}= \eta(f-\gamma).
\end{equation}
Letting in the last limit $\gamma=-\frac{e(\hat{x})}{\ee^{\hat{x}}\left\{\hat{\tau}\right\}}$ we obtain
\[\lim_{T\to \infty}{1 \over T} \ee^{\hat{x}} \Big\{ \int_0^{T}  f(X^*_s)ds + \sum_{i=1}^\infty  \ind{\tau_i^* \le T} e(\hat{x}) \Big\}={\ee^{\hat{x}}\left\{\int_0^{\hat{\tau}} f(X_s)ds + e(\hat{x})\right\} \over \ee^{\hat{x}}\left\{\hat{\tau}\right\}}=\lambda.\]

In the case of a general cost function and $w$ bounded from above, we obtain from iterated application of Bellman equation \eqref{eq22} and Lemma \ref{lem:super_mart}
\begin{equation}\label{eq:optfunc}
w(x) + \lambda T = \ee^x \Big\{ \int_0^{\tau^*_n \wedge T} f(X^*_s) ds + \sum_{i=1}^{n-1} \ind{\tau^*_i \le T} c(X^n_{\tau^*_i-}, \xi^*_i) + w(X^*_{\tau^*_n \wedge T})\Big\}.
\end{equation}
There is a finite number of impulses before time $T$, $\prob^x$-a.s., because $c (x, \xi) \le c < 0$ and $f$ and $w$ are bounded from above. Hence $\lim_{i \to \infty} \tau^*_i = \infty$, $\prob^x$-a.s.
Passing to the limit with $n$ using Fatou's lemma and boundedness from above of all terms under expectation and dividing both sides by $T$ yields
\[
\frac{w(x)}{T} + \lambda \le \frac{1}{T} \ee^x \Big\{ \int_0^T f(X^*_s) ds + \sum_{i=1}^\infty \ind{\tau^*_i \le T} c(X^n_{\tau^*_i-}, \xi^*_i) + w(X^*_T)\Big\}.
\]
Taking $\liminf_{T \to \infty}$ on both sides completes the proof of optimality of $(\tau^*_i, \xi^*_i)$ provided that one shows that $\liminf_{T \to \infty} \ee^x \{ w(X^*_T) / T \} \le 0$ and this is in the case because $w$ is bounded from above.
\end{proof}

Boundedness of $w$, required above for proving the optimality of $(\tau^*_i, \xi^*_i)$ for the functional $J$, is established in the following lemma.
\begin{lemma}\label{lem:optim_strat}
Assume \ref{ass:weak_feller}, \ref{ass:speed_ergodic}. If $\mu(f)<\liminf_{\|x\|\to \infty} f(x)$ then $q$ is bounded from below. Assume additionally \ref{ass:cost_function}-\ref{ass:unifintc_b}, and $\limsup_{\alpha \to 0}\alpha v_\alpha(z)=\lambda>\mu(f)$. If $ \limsup_{\|x\|\to \infty}f(x)< \lambda$ or $q$ is bounded from above then  $w$ is bounded from above.
\end{lemma}
\begin{proof}
When $\mu(f)<\liminf_{\|x\|\to \infty} f(x)$, boundedness from below of $q$ follows from Lemma \ref{lem:bounded from below}. If $\limsup_{\|x\|\to \infty}f(x)< \lambda$ the set $F=\left\{x: f(x)\geq \lambda\right\}$ is compact. Exploiting that in \eqref{eq22} one may take bounded stopping times,  Lemma \ref{lem:super_mart} implies
\begin{eqnarray*}
w(x)
&\le &\sup_{\tau \text{-bounded}} \ee^x \Big\{\int_0^{\tau \wedge D_{F\cup U}}(f(X_s)-\lambda)ds + \ind{\tau\leq D_{F \cup U}} Mw(X_\tau) + \ind{\tau> D_{F\cup U}}w(X_{D_{F \cup U}})\Big\}\\
&\le& \sup_{y\in U\cup F}w(y),
\end{eqnarray*}
which means that $w$ is bounded form above. If $q$ is bounded from above then by \eqref{wbounds} and Lemma \ref{lem:bounded from below} $w$ is also bounded from above.
\end{proof}

\begin{theorem}\label{donothing} Under \ref{ass:weak_feller}, \ref{ass:speed_ergodic}, \ref{ass:positrec} when $\limsup_{\alpha \to 0}\alpha v_\alpha(z)=\mu(f)$ we have that $\alpha v_\alpha(x)\to \mu(f)$, as $\alpha \to 0$,  uniformly in $x$ from compact sets and the strategy `do nothing' is optimal for the functional $J$.
\end{theorem}
\begin{proof}
By Lemma \ref{lem:from_above} we have $\lim_{\alpha \to 0}\alpha v_\alpha(z)=\mu(f)$. Assume that there is a sequence $x_n\in U$ and $\alpha_n\to 0$  for which
\(\lim_{n \to \infty}\alpha_n v_{\alpha_n}(x_n)>\mu(f).\)
Then
\[\lim_{n \to \infty}\alpha_n v_{\alpha_n}(z)\geq \lim_{n \to \infty}\alpha_n \big(c(z,x_n)+ v_{\alpha_n}(x_n)\big)>\mu(f)\]
a contradiction. Combining it with Lemma \ref{lem:from_above} proves
\(
%\label{eq:39}
\lim_{\alpha \to 0}\sup_{y\in U}\alpha v_\alpha(y)=\mu(f).
\)
Assume now that a sequence $x_n$ is from an arbitrary compact set $\Gamma$. Assumption \ref{ass:positrec} yields $\sup_n \ee^{x_n}\{D_U\}<\infty$, and this is also true when the trajectory is controlled as then the process may enter $U$ even earlier because each impulse shifts to $U$. Therefore,
\[
\lim_{n \to \infty}\alpha_n v_{\alpha_n}(x_n)\leq \lim_{n \to \infty}\alpha_n \big(\|f\| \sup_{y\in \Gamma} t(y) + \sup_{x\in U} v_{\alpha_n}(x)\big)=\mu(f),
\]
which together with Lemma \ref{lem:from_above} gives $\alpha_n v_{\alpha_n} (x_n) \to \mu(f)$. In fact, the latter argument proves uniform on compact sets convergence of $\alpha v_\alpha(x)$ to $\mu(f)$. Now, by Lemma \ref{lem:tauber}
\begin{align*}
&\liminf_{T \to \infty} \frac{1}{T} \ee^x \Big\{ \int_0^T f(X_s) ds + \sum_{i = 1}^\infty \ind{\tau_i \le T} c(X_{\tau_i-}, \xi_i) \Big\}\\
&\le \liminf_{\alpha\to 0} \alpha \ee^x \Big\{ \int_0^\infty e^{-\alpha s} f(X_s) ds + \sum_{i=1}^\infty \ind{\tau_i < \infty} e^{-\alpha \tau_i} c(X_{\tau_i-}, \xi_i) \Big\}
\le \liminf_{\alpha\to 0} \alpha v_\alpha(x)=\mu(f).
\end{align*}
Since $\sup_{(\tau_i, \xi_i)} J\big(x, (\tau_i, \xi_i)\big)\ge \mu(f)$, it is clear that the strategy `do nothing' is optimal.
\end{proof}

\section{Relaxation of assumption on \texorpdfstring{$q$}{q}}
%\section{Approximations of impulse control functional}
In the previous section we required that the function $f$ is such that its potential $q$ is bounded from below, and we constructed an optimal strategy when $w$ was bounded from above, c.f. Lemma \ref{lem:optim_strat}. We shall now approximate a general continuous bounded $f$ by functions with potentials bounded from below and corresponding $w$ being bounded above. Without loss of generality we can restrict ourselves to functions $f$ which are nonnegative. We also assume that $f$ is not constant $\mu$-a.s.; otherwise the control problem is trivial. The main result of this section is Theorem \ref{thm:bar_lambda} which  shows that the optimal value of the functional \eqref{eqn:main_functional} for a general continuous bounded $f$ does not depend on $x$ and provides explicit construction of $\ve$-optimal control strategies.

Let $B_{z,N}$ be a ball with center in $z$ and radius $N$ ($z \in U$ is the point fixed in the previous section for the definition of $w_\alpha$). For $\eta \in (\mu(f), \|f\|)$ define
\begin{equation}
f_N(x)=f(x)(1-\rho(x,B_{z,N}))^+ + \eta(1-\rho(x,B_{z,N+1}^c))^+.
\end{equation}

\begin{lemma} We have $\|f_N\| \le \|f\|$ and $\lim_{N \to \infty} \mu(f_N) = \mu(f)$. For sufficiently large $N$ the set $\left\{x: f_N(x)\leq \mu(f_N)\right\}$ is contained in $B_{z,N+1}$.
\end{lemma}
\begin{proof}
The bound for the norm of $f_N$ follows easily from the definition. Then  $\mu(f_N) = \mu(f) + \mu(f_N - f) \le  2\|f\| \mu (B_{z, N+1}) \to 0$ as $N \to \infty$. The remaining claim of the lemma is now obvious.
\end{proof}

For an admissible impulse strategy $V=(\tau_i, \xi_i)$ and $\delta \in (0, -c)$ we define three functionals
\begin{align*}
J^{N,c}\big(x, V\big) &= \liminf_{T \to \infty} \frac{1}{T} \ee^x \Big\{ \int_0^T f_N(X_s) ds + \sum_{i = 1}^\infty \ind{\tau_i \le T} c(X_{\tau_i-}, \xi_i) \Big\},\\
J^{N,c+\delta}\big(x, V\big) &= \liminf_{T \to \infty} \frac{1}{T} \ee^x \Big\{ \int_0^T f_N(X_s) ds + \sum_{i = 1}^\infty \ind{\tau_i \le T} \big(c(X_{\tau_i-}, \xi_i)+\delta\big) \Big\},\\
J^{N,\delta}\big(x, V\big) &= \liminf_{T \to \infty} \frac{1}{T} \ee^x \Big\{ \int_0^T  (f(X_s)-f_N(X_s)) ds - \sum_{i = 1}^\infty \ind{\tau_i \le T}\delta \Big\},
\end{align*}
related value functions
\begin{align*}
\bar{\lambda}(x)&=\sup_V J\big(x,V\big), &\bar{\lambda}^{N,c}(x)&=\sup_V J^{N,c}\big(x, V \big),\\
\bar{\lambda}^{N,c+\delta}(x)&=\sup_V J^{N,c+\delta}\big(x, V\big), &\bar{\lambda}^{N,\delta}(x)&=\sup_V J^{N,\delta}\big(x, V \big),
\end{align*}
and discounted value functions
\begin{align*}
v^{N,c}_\alpha (x) &= \sup_V \ee^x \Big\{ \int_0^\infty e^{-\alpha s} f_N(X_s) ds + \sum_{i=1}^\infty \ind{\tau_i < \infty} e^{-\alpha \tau_i} c(X_{\tau_i-}, \xi_i) \Big\},\\
v^{N,c+\delta}_\alpha (x) &= \sup_V \ee^x \Big\{ \int_0^\infty e^{-\alpha s} f_N(X_s) ds + \sum_{i=1}^\infty \ind{\tau_i < \infty} e^{-\alpha \tau_i} \big(c(X_{\tau_i-}, \xi_i) + \delta)\Big\},\\
v^{N,\delta}_\alpha (x) &= \sup_V \ee^x \Big\{ \int_0^\infty e^{-\alpha s} \big(f(X_s) -f_N(X_s)\big)ds - \delta \sum_{i=1}^\infty e^{-\alpha \tau_i}) \Big\},
\end{align*}
with $v_\alpha$ defined in Section \ref{sec:impulse_control}.

The introduction of the cost $\delta > 0$ is only for technical reasons so that we can use results from previous sections to characterise $\bar{\lambda}^{N,\delta}(x)$ which evaluates the difference between two running costs. We will prove that $\lim_{N \to \infty} \bar{\lambda}^{N,\delta}(x) = 0$ which, intuitively, should hold for $\delta = 0$ for most ergodic processes. Indeed, impulses can only shift the process to a compact set $U$ and uncontrolled process will spend little time in the complement of a sufficiently large ball $B^c_{z, N}$ as $\mu(B^c_{z, N}) \to 0$ as $N \to \infty$. Providing an accurate proof of this fact is beyond the scope of this paper and we will assume $\delta > 0$.

Our goal now is to choose such $\eta$ in the definition of $f_N$ that $\eta < \limsup_{\alpha \to 0} \alpha v^{N,c}_\alpha (z)$, i.e., by Lemma \ref{lem:optim_strat} and Theorem \ref{optimalityoflambda} function $\bar\lambda^{N,c}(x)$ is constant and there is a strategy realising this value. This will be further used to show that $\bar\lambda$ is constant and to provide $\ve$-optimal strategies for the functional $J$. So far we can only establish that $\bar\lambda$ is constant on $U$ and this value is a lower bound for $\bar \lambda$ on $E$.
\begin{lemma}
Function $\bar\lambda$ is constant on $U$ and $\bar\lambda(x) \ge \bar\lambda(z)$ for any $x \in E$.
\end{lemma}
\begin{proof}
Take $x,y \in U$. Then $J(x, V) \ge \lim_{T \to \infty} \frac{c(x, y)}{T} + J(y, V) = J(y, V)$. By symmetry we obtain the equality. Similarly, for any $x \in E$ we have $J(x, V) \ge J(z, V)$.
\end{proof}

The following assumption will play a key role in establishing that $\lim_{N \to \infty} \bar{\lambda}^{N,\delta}(x) = 0$. A sufficient condition is discussed in Remark \ref{rem:E1}.
\begin{assumption}
\item[(C)]\label{ergodicstopping}
\( \displaystyle
\sup_{x\in U} \sup_\tau \frac{\ee^x \big\{ \int_0^{\tau} \ind{B^c_{z,N}}(X_s) ds\big\}}{\ee^x \{\tau\}} \to 0 \qquad \text{as $N\to \infty$.}
\)
\end{assumption}
Define
\begin{equation}
\tl f_N(x)=f(x)(1-\rho(x,B_{z,N}))^+ + (\|f\|+1)(1-\rho(x,B_{z,N+1}^c))^+.
\end{equation}

\begin{lemma} We have $\tl f_N(x)\geq f_N(x)$ for $x\in E$, $\|\tl f_N\|=\|f\|+1$ and the set $\{x: \tl f_N(x)\leq \mu(\tl f_N)\}$ is contained in $B_{z,N+1}$.
\end{lemma}

 \begin{proof}
The first claim is obvious. Since $\mu(\tl f_N)\leq \mu(B_{z,N})\|f\| + \mu(B_{z,N}^c)(\|f\|+1)=\|f\| + (1-\mu(B_{z,N})) < \|f\| + 1$ we have that
\[
\left\{x:\tl  f_N(x)\leq \mu(\tl f_N)\right\} \subset \left\{x: \tl f_N(x) < \|f\| + 1\right\}\subset B_{z,N+1}.
\]
\end{proof}

\begin{lemma}\label{valuelambda}
Under assumptions \ref{ass:weak_feller}, \ref{ass:speed_ergodic} and \ref{ass:positrec} we have
\begin{equation}\label{value1-}
\mu(f - f_N) \le \bar{\lambda}^{N,\delta}(x) \le \tl \lambda^{N,\delta},
\end{equation}
where either
\begin{equation}\label{value1}
\tilde{\lambda}^{N,\delta}=\sup_{x\in U} \sup_\tau {\ee^x\left\{\int_0^\tau (\tilde f_N(X_s)-f_N(X_s))ds - \delta\right\} \over \ee^x\left\{\tau\right\}},
\end{equation}
or
\begin{equation}\label{value2}
\tilde{\lambda}^{N,\delta}=\mu(\tl f_N-f_N).
\end{equation}
Moreover,
\begin{equation}\label{value2+}
\mu(f - f_N) \le \liminf_{\alpha \to 0} \alpha v_\alpha^{N, \delta}(z) \le \limsup_{\alpha \to 0} \alpha v_\alpha^{N, \delta}(z) \le \tilde{\lambda}^{N,\delta}.
\end{equation}
If, in addition, we assume \ref{ergodicstopping} then $\lim_{N \to \infty} |\bar{\lambda}^{N,\delta} (x)|= 0$ uniformly in $x \in E$ and $\lim_{N \to \infty} \tl{\lambda}^{N,\delta}= 0$ for any $\delta > 0$.
\end{lemma}

\begin{proof}
Notice first that the set $\{x: \tl f_N(x)-f_N(x)\leq \mu(\tl f_N-f_N)\}$ is compact for a sufficiently large $N$. Indeed,
\(
\tl f_N(x)-f_N(x) = (\|f\|+1 - \eta)(1-\rho(x,B_{z,N+1}^c))^+
\)
and $\mu(\tl f_N - f_N) < \|f\|+1 - \eta$ because $\mu(B_{z,N} > 0$ for $N$ large enough.

Lemma \ref{lem:bounded from below} implies that assumption \ref{ass:boundedpotential} is satisfied with $f$ replaced by $\tl f_N-f_N$. Let $v_\alpha^N(x)$ be the optimal value of the discounted functional
\[
\ee^x \Big\{ \int_0^\infty e^{-\alpha s} \big(\tl f_N(X_s)-f_N(X_s)\big) ds - \sum_{i = 1}^\infty e^{-\alpha \tau_i} \delta \Big\}
\]
and
\[
\tl\lambda^{N, \delta} (x) = \sup_V \liminf_{T \to \infty} \frac1T \ee^x \Big\{ \int_0^T \big(\tl f_N(X_s)-f_N(X_s)\big) ds - \delta \sum_{i = 1}^\infty \ind{\tau_i \le T} \Big\}.
\]
If $\limsup_{\alpha \to 0} \alpha v^N_\alpha(z)>\mu(\tl f_N-f_N)$ then by Proposition \ref{propseparated}, $\tilde{\lambda}^{N,\delta}$ is constant and of the form \eqref{value1}. In the case when $\limsup_{\alpha \to 0} \alpha v^N_\alpha(z)=\mu(\tl f_N-f_N)$, Theorem \ref{donothing} yields \eqref{value2}. Inequalities \eqref{value1-} now follow easily. Since $v_\alpha^{N, \delta}(z) \le v^N_\alpha (z)$ and $\limsup_{\alpha \to 0} \alpha v^N_\alpha(z) = \tl\lambda^{N, \delta}$, the right-hand inequality in \eqref{value2+} holds. Lemma \ref{lem:from_above} yields the left-hand inequality.

We have that either $\tilde{\lambda}^{N,\delta}=\mu(\tl f_N-f_N)\to 0$ as $N\to \infty$ or, from \eqref{value1},
\[
\tilde{\lambda}^{N,\delta}
\leq \sup_{x\in U}\sup_\tau \frac{(2\|f\| + 1)\ee^x \Big\{ \int_0^{\tau} \ind{B^c_{z,N}}(X_s) ds\Big\}}{\ee^x \{\tau\}} \to 0 \qquad \text{as $N\to \infty$}
\]
by assumption \ref{ergodicstopping}. Easily, $\lim_{N \to \infty} \mu(f - f_N) = 0$.
\end{proof}

\begin{remark}\label{rem:E1}
Notice that assumption  \ref{ergodicstopping} is satisfied when $(X_t)$ is a continuous process and
\begin{equation}
\lim_{N \to \infty} \frac{\sup_{x\in \partial B_{z,N}^c} t(x)}{\inf_{x\in U} \ee^x\{D_{B_{z,N}^c}\}} = 0,
\end{equation}
which means that the process returns to the set $U$ from the boundary of $B_{z,N}^c$ much faster that enters this set starting from $U$. Such property is typical for ergodic processes and is usually attained by a suitable Lyapunov condition. 
%Furthermore, under \ref{ergodicstopping} using Theorem \ref{approxofstra} we have that optimal strategies for the functional with the running reward $f$ can be approximated by nearly optimal strategies corresponding to functional with running reward $f_N$, for which by Theorem \ref{thm:bellman} we have Bellman equation (since the potentials are bounded form below) and all consequences following from it.
\end{remark}

\begin{prop}\label{prop:f_N}
Assume \ref{ass:weak_feller}, \ref{ass:speed_ergodic}, \ref{ass:cost_function}-\ref{ass:positrec}, \ref{ergodicstopping} and $\bar\lambda(z) > \mu(f)$. There is $\eta \in (\mu(f), \|f\|)$ to be used in the definition of $f_N$ such that $\{ x \in E:\ f_N(x) \le \mu(f_N) \}$ is compact and $\limsup_{\|x\| \to \infty} f_N(x) \allowbreak < \bar \lambda^{N,c}(z)$ for a sufficiently large $N$. Furthermore, $\bar\lambda^{N,c}$ is constant and the strategy given in Theorem \ref{thm:bellman} is optimal.
\end{prop}
\begin{proof}
By Tauberian theorem, Lemma \ref{lem:tauber}, $\limsup_{\alpha \to 0} \alpha v_\alpha(z) \ge \bar\lambda(z) > \mu(f)$. We will show that for any $\ve > 0$ we have $\liminf_{\alpha \to 0} \alpha v_\alpha^{N,c}(z) \ge \bar\lambda(z) - \ve$ for a sufficiently large $N$ and any $\eta \in (\mu(f), \|f\|)$. For any $\delta>0$ such that $c+\delta<0$ (recall that the constant $c < 0$ is an upper bound on the cost function $c$) we have:
\[
\limsup_{\alpha \to 0} \alpha \big(v_\alpha(z) - v_\alpha^{N,c}(z) \big)
\le
\limsup_{\alpha \to 0} \alpha \big(v_\alpha(z) - v_\alpha^{N,c+\delta}(z) \big) + \limsup_{\alpha \to 0} \alpha \big(v_\alpha^{N, c+\delta}(z) - v_\alpha^{N,c}(z) \big).
\]
Since
\(
v_\alpha(z) - v_\alpha^{N,c+\delta}(z) \le v^{N, \delta}_\alpha (z)%\sup_V \ee^x \Big\{ \int_0^\infty \big(f(X_s) - f_N(X_s)\big) ds - \delta \sum_{i=1}^\infty e^{-\alpha \tau_i} \Big\}
\),
Lemma \ref{valuelambda} implies that $\limsup_{\alpha \to 0} \alpha \big(v_\alpha(z) - v_\alpha^{N,c+\delta}(z)\big) \le \tl\lambda^{N, \delta}$. To bound the second term we study an estimate for $\ve$-optimal strategies for $v^{N, c+\delta}_\alpha (z)$. Fix such a strategy. Then
\begin{align*}
v^{N, c+\delta}_\alpha(z) - \ve
&\le \ee^x \Big\{ \int_0^\infty e^{-\alpha s} f_N(X_s) ds + \sum_{i=1}^\infty \ind{\tau_i < \infty} e^{-\alpha \tau_i} \big(c(X_{\tau_i-}, \xi_i) + \delta\big) \Big\}\\
&\le \frac{\|f_N\|}{\alpha} + (c + \delta) \ee^x \Big\{ \sum_{i=1}^\infty e^{-\alpha \tau_i} \Big\},
\end{align*}
and
for $0\leq \delta<-c$
\begin{equation}\label{eqn:b1}
\ee^x \Big\{ \sum_{i=1}^\infty e^{-\alpha \tau_i} \Big\} \le \frac{\frac{\|f_N\|}{\alpha} - v_\alpha^{N, c+\delta}(z) + \ve}{-c - \delta}\le \frac{\frac{\|f\|}{\alpha} - v_\alpha^{N, c+\delta}(z) + \ve}{-c - \delta}.
\end{equation}
Hence, in the definition of $v^{N, c+\delta}_\alpha$  strategies can be assumed to satisfy the above bound. This allows us to compute the bound
\[
\alpha\big(v_\alpha^{N, c+\delta}(z) - v_\alpha^{N, c}(z)\big) \le \alpha \sup_V \ee^x \Big\{ \delta \sum_{i=1}^\infty e^{-\alpha \tau_i} \Big\} \le \delta \frac{\|f\| - \alpha v_\alpha^{N, c+\delta}(z) + \alpha \ve}{-c - \delta},
\]
where supremum is over strategies satisfying \eqref{eqn:b1}. Using the lower bound $v_\alpha^{N, c+\delta} (z) \ge \frac{-\|f\|}{\alpha}$ we obtain
\(
\limsup_{\alpha \to 0} \alpha\big(v_\alpha^{N, c+\delta}(z) - v_\alpha^{N, c}(z)\big) \le \delta \frac{2\|f\|}{-c-\delta}.
\)
Hence,
\[
\limsup_{\alpha \to 0} \alpha \big(v_\alpha(z) - v_\alpha^{N,c}(z) \big) \le \tl\lambda^{N, \delta} + \delta \frac{2\|f\|}{-c-\delta},
\]
which gives
\[
\liminf_{\alpha \to 0} \alpha v_\alpha^{N,c}(z) \ge \limsup_{\alpha \to 0} \alpha v_\alpha(z) - \tl\lambda^{N, \delta} - \delta \frac{2\|f\|}{-c-\delta} \ge \bar\lambda(z) - \tl\lambda^{N, \delta} - \delta \frac{2\|f\|}{-c-\delta}.
\]
The last term can be made arbitrarily small by the choice of $\delta$ sufficiently close to $0$. By Lemma \ref{valuelambda}, $\tl\lambda^{N, \delta}$ can be made arbitrarily close to $0$ for $N$ sufficiently large. 
Take $\ve < \bar\lambda(z) - \mu(f)$ and choose $N^*, \delta^*$ such that $\tl{\lambda}^{N,\delta} < \ve/2$ for all $N \ge N^*$ and $\delta \frac{2\|f\|}{-c-\delta} < \ve/2$ for $\delta < \delta^*$ (notice that the choice of $N^*, \delta^*$ holds uniformly in $\eta \in (\mu(f), \|f\|)$). Now, choose $\eta \in (\mu(f), \bar\lambda(z) - \ve)$. Recalling $\lim_{N \to \infty} \mu(f_N) = \mu(f)$, we obtain that $\mu(f_N) < \eta$ for sufficiently large $N$. Hence $\{x \in E: f_N(x) \le \mu(f_N)\} \subset B_{z, N+1}$ for large enough $N$, so it is compact, and
\[
\limsup_{N \to \infty} f_N(x) = \eta   < \bar\lambda(z) - \ve \le
\liminf_{\alpha \to 0} \alpha v_\alpha^{N,c}(z) \le \limsup_{\alpha \to 0} \alpha v_\alpha^{N,c}(z).
\]
Let  $\lambda = \limsup_{\alpha \to 0}\alpha v_\alpha^{N,c}(z)$. Clearly $\lambda > \mu(f_N)$. Lemma \ref{lem:optim_strat}, Theorem \ref{optimalityoflambda} and Theorem \ref{thm:bellman} imply that the strategy derived from the Bellman equation in Theorem \ref{thm:bellman} is optimal for $J^{N,c}$ and $\bar\lambda^{N,c}(x) = \lambda$ for all $x$.
%In the notation of Theorem \ref{thm:bellman}, $\lambda = \limsup_{\alpha \to 0} \alpha v_\alpha^{N,c}(z)$ for the functional $J^{N,c}$ and clearly $\lambda > \mu(f_N)$. Lemma \ref{lem:optim_strat} and Theorem \ref{optimalityoflambda} imply that the strategy derived from the Bellman equation in Theorem \ref{thm:bellman} is optimal and $\bar\lambda^{N,c}(x) = \lambda$ for all $x$.
\end{proof}
\begin{cor}
By the same arguments as at the end of the proof of Proposition \ref{prop:f_N}, the value function $\bar \lambda^{N, c+\delta}$ is constant on $E$ and the strategy from Theorem \ref{thm:bellman} is optimal.
\end{cor}

\begin{theorem}\label{thm:bar_lambda}
Assume \ref{ass:weak_feller}, \ref{ass:speed_ergodic}, \ref{ass:cost_function}-\ref{ass:positrec}, \ref{ergodicstopping} and $\bar\lambda(z) > \mu(f)$ for some $z \in E$. Then the function $\bar \lambda$ is constant on $E$ and for any $\ve > 0$ there is $N$ such that an optimal strategy for $\bar \lambda^{N,c}$ is $\ve$-optimal for $\bar \lambda$.
\end{theorem}

The proof of the above theorem is split into several lemmas. Let $N^V(0,T)$ be the number of impulses under control $V$ in the time interval $[0,T]$.

\begin{lemma}\label{lem:strbound}
For the value function $\bar{\lambda}(x)$ (respectively, $\bar{\lambda}^{N,c}$), the strategies can be restricted to those satisfying
\begin{equation}\label{strbound}
\limsup_{T\to \infty} {1 \over T}\ee^x\left\{N^V(0,T)\right\}\leq \frac{\|f\|+ \ve}{-c} \qquad \Big(\le \frac{2\|f\| + \ve}{-c} \Big)
\end{equation}
for a fixed $\ve > 0$. For $\bar{\lambda}^{N,c+\delta}$ the bound changes to $\frac{2\|f\| + \ve}{-c-\delta}$ provided that $c+\delta<0$.
\end{lemma}
\begin{proof} Consider the functional $J\big(x, V\big)$. For $\ve>0$, any $\ve$-optimal strategy $V_\ve=(\tau_i, \xi_i)$ satisfies
\begin{align*}
 \bar{\lambda}(x)-\ve
 &\leq \liminf_{T \to \infty} \frac{1}{T} \ee^x \Big\{ \int_0^T f(X_s) ds + \sum_{i = 1}^\infty \ind{\tau_i \le T} c(X_{\tau_i-}, \xi_i) \Big\}\\
 &\leq \|f\|  + c \limsup_{T \to \infty}\frac{1}{T} \ee^x\left\{N^{V_\ve}(0,T)\right\}.
\end{align*}
Since one can obviously constrain the supremum defining $\bar\lambda(x)$ to $\ve$-optimal strategies and taking into account that $\bar{\lambda}(x)\geq 0$, we obtain \eqref{strbound}. For the other two functionals, we use the lower bound $\min\big(\bar{\lambda}^{N,c},  \bar{\lambda}^{N,c+\delta}\big) \ge -\|f\|$ instead of $0$, and for $\bar{\lambda}^{N,c+\delta}$ use the upper bound on the cost equal to $c + \delta$.
\end{proof}

\begin{lemma}\label{approxofstra}
\begin{equation}\label{fact1}
- \bar{\lambda}^{N,\delta} (x) - \delta {\|f\| \over -c} \leq \bar{\lambda}^{N,c}-\bar{\lambda}(x)\leq \tl{\tl{\lambda}}^{N,\delta} + \delta \frac{2\|f\|}{-c},
\end{equation}
where either
\begin{equation}\label{value1p}
\tl{\tl{\lambda}}^{N,\delta}=\sup_{x\in U} \sup_\tau {\ee^x\left\{\int_0^\tau (\tilde f_N(X_s)-f(X_s))ds - \delta\right\} \over \ee^x\left\{\tau\right\}},
\end{equation}
or
\begin{equation}\label{value2p}
\tl{\tl{\lambda}}^{N,\delta}=\mu(\tl f_N-f),
\end{equation}
and both converge to $0$ when $N \to \infty$.
%If $\lim_{N \to \infty} \tilde{\lambda}^{N,\delta} = 0$ for each $\delta>0$, then $\bar{\lambda}^{N,c}(x)\to\bar{\lambda}(x)$, as $N\to \infty$ and an $\ve$-optimal strategy for $\bar{\lambda}^{N,c}(x)$ is $2\ve$ optimal for $\bar{\lambda}(x)$ when $N$ is sufficiently large.
\end{lemma}
\begin{proof}
We start from the lower bound in \eqref{fact1}. Notice that
\begin{align*}
 \bar{\lambda}(x) - \bar{\lambda}^{N,c+\delta}
 &\le \sup_V \left( J(x,V) - J^{N,c+\delta}(x, V) \right)\\
 &\le \sup_V \limsup_{T \to \infty} \frac{1}{T} \ee^x \Big\{ \int_0^T  \big(f(X_s) - f_N(X_s)\big) ds - \sum_{i = 1}^\infty \ind{\tau_i \le T}\delta \Big\}
 =\bar{\lambda}^{N,\delta} (x),
 \end{align*}
where the last equality follows from Proposition \ref{propseparated} and Lemma \ref{valuelambda} which imply integrability of impulse times in optimal strategy for $\bar{\lambda}^{N,\delta}(x)$ and the monotone convergence theorem which implies equivalence of functionals with $\limsup$ and $\liminf$ under this integrability condition.

For any $\ve > 0$ and a strategy $V_\ve$ that is $\ve$-optimal for $\bar{\lambda}(x)$, we have
 \begin{align}\label{estvalue_b}
\bar{\lambda}(x) - \bar{\lambda}^{N,c}
&\le J(x,V_\ve) - J^{N,c}(x, V_\ve)  + \ve \nonumber \\
&\le J(x,V_\ve) -J^{N,c+\delta}(x, V_\ve)+ J^{N,c+\delta}(x, V_\ve) - J^{N,c}(x, V_\ve)   + \ve\nonumber \\
&\le \bar{\lambda}^{N,\delta} (x) + \delta {\|f\|+ \ve \over -c} +  \ve,
\end{align}
where the second term in \eqref{estvalue_b} follows from the fact that we are allowed to restrict ourselves to strategies $V$ satisfying \eqref{strbound} while the first term results from the calculations in the preceding part of the proof. From arbitrariness of $\ve$ we obtain \eqref{fact1}.

For the upper bound in \eqref{fact1}, take an $\ve$-optimal strategy $V_\ve$ for $\bar{\lambda}^{N,c}$ satifying \eqref{strbound}. Then
 \begin{equation}\label{estvalue}
\bar{\lambda}^{N,c}-\bar{\lambda}(x)\le  J^{N,c}(x, V_\ve)-J^{N,c+\delta}(x, V_\ve)+ J^{N,c+\delta}(x, V_\ve)- J(x,V_\ve)  + \ve.
\end{equation}
Since the cost $c+\delta$ is smaller than the cost $c$, we have
\(
J^{N,c}(x, V_\ve)-J^{N,c+\delta}(x, V_\ve) \le 0.
\)
Esimation of the other difference is more involved:
\begin{align*}
J^{N,c+\delta}(x, V_\ve)- J(x,V_\ve)
&\le \limsup_{T \to \infty} \frac{1}{T} \ee^x \Big\{ \int_0^T  (f_N(X_s)-f(X_s)) ds + \sum_{i = 1}^\infty \ind{\tau_i \le T}\delta \Big\} \\
 &\le \limsup_{T \to \infty} \frac{1}{T} \ee^x \Big\{ \int_0^T  (f_N(X_s)-f(X_s)) ds - \sum_{i = 1}^\infty \ind{\tau_i \le T}\delta \Big\} \\
 &\hspace{11pt}+ \limsup_{T \to \infty} \frac{1}{T} \ee^x \Big\{ 2 \delta \sum_{i = 1}^\infty \ind{\tau_i \le T} \Big\}
 \le \tl{\tl{\lambda}}^{N,\delta} + \delta \frac{2\|f\|+ \ve}{-c},
 \end{align*}
where the first term is requires a proof identical as in Lemma \ref{valuelambda} and the reasoning about integrability of impulse times as for the lower bound. The second term follows from Lemma \ref{lem:strbound}. Inserting above estimates into \eqref{estvalue} and exploiting the arbitratiness of $\ve$ proves the upper bound in \eqref{fact1}. The claim about convergence of $\tl{\tl{\lambda}}^{N, \delta}$ to zero requires identical arguments as in the proof of Lemma \ref{valuelambda}.
 \end{proof}

 \begin{proof}[Proof of Theorem \ref{thm:bar_lambda}]
 The upper and lower bound in \eqref{fact1} do not depend on $x$ and can be made arbitrarily small by choosing sufficiently large $N$ and small $\delta$. This proves $\lim_{N \to \infty} \bar\lambda^{N,c} = \bar \lambda(x)$, so $\bar\lambda$ is a constant function. Clearly, for any $\ve > 0$ there is $N$ such that $|\bar\lambda^{N,c} - \bar \lambda(x)| \le \ve$ and an optimal strategy for $\bar\lambda^{N,c}$ is $\ve$-optimal for $\bar \lambda$.
\end{proof}

\begin{cor}
Under assumptions of Theorem \ref{thm:bar_lambda}, the value function $\hat \lambda (x) = \sup_V \hat J(x, V)$, where the supremum is over integrable strategies, coincides with $\bar \lambda$ and optimal strategies for $\bar \lambda^{N,c}$ are integrable and $\ve$-optimal for $\hat \lambda$.
\end{cor}

We will present below a significantly shorter proof that $\hat\lambda(x)$ does not depend on $x$ and there are $\ve$-optimal strategies under the same assumptions as those of Theorem \ref{thm:bar_lambda}. Let \begin{equation}
\hat{J}^{N,c}\big(x, V\big) = \liminf_{n \to \infty} \frac{1}{\ee^x\left\{\tau_n\right\}} \ee^x \Big\{ \int_0^{\tau_n} \tl f_N(X_s) ds + \sum_{i = 1}^n  c(X_{\tau_i-}, \xi_i) \Big\},
\end{equation}
$\hat \lambda_N(x)=\sup_V \hat{J}^{N,c}\big(x, V\big)$ and $\hat{\lambda}(x)=\sup_V \hat{J}\big(x, V\big)$, where the suprema are taken over integrable strategies. Theorem \ref{optimalityoflambda} implies that $\hat \lambda_N$ does not depend on $x$ while the following theorem will prove it for $\hat \lambda$.
\begin{theorem}\label{thm:final}
Under \ref{ass:weak_feller}-\ref{ass:speed_ergodic}, \ref{ass:cost_function}-\ref{ass:positrec}, \ref{ass:unifintc_b} and  \ref{ergodicstopping} if  $\limsup_{\alpha \to 0}\alpha v_\alpha(z)\allowbreak =\lambda>\mu(f)$ we have that $\hat \lambda_N$ is a constant function for sufficiently large $N$ and $\hat \lambda_N(x) \to \hat{\lambda}(x)$, as $N\to \infty$, so, in particular, $\hat \lambda$ is a constant function.
\end{theorem}
\begin{proof}
We assume without loss of generality that $f \ge 0$. Since $\tl f_N\geq f$ and $\lim_{N \to \infty} \mu(\tl f_N)= \mu(f)$, then $\limsup_{\alpha \to 0}\alpha \hat v^{N,c}_\alpha(z)>\mu(\tl f_N)$ for sufficiently large $N$, where $\hat v^{N,c}_\alpha$ is the analogue of $v^{N,c}_\alpha$ with $f_N$ replaced by $\hat f_N$. Therefore, by Theorem \ref{thm:bellman} and \ref{optimalityoflambda} $\hat\lambda_N(x)$ does not depend on $x$.
% and there exists a continuous function $w^N$ and a constant $\hat \lambda_N$ such that the Bellman equation
% \begin{equation}
% w^N(x)=\sup_\tau \ee^x\left\{\int_0^\tau (f_N(X_s)-\hat \lambda_N)ds + Mw^N(X_\tau)\right\}
% \end{equation}
% is satisfied. Theorem \ref{optimalityoflambda} implies that $\hat \lambda_N$ is the optimal value of the functional $\hat{J}^{N,c}(x,V)$.
For any integrable strategy $V$, we have
\begin{align*}
0&\leq \hat{J}^{N,c}\big(x, V\big)- \hat{J}\big(x, V\big) \le \limsup_{n\to \infty} \frac{1}{\ee^x\{\tau_n\}} \ee^x\Big\{\int_0^{\tau_n}\big(f_N(X_s) - f(X_s)\big) ds \Big\} \\
&\leq \limsup_{n\to \infty}\frac{\ee^x\big\{\int_0^{\tau_n}(\|f\|+1)\ind{B_{z,N}^c}(X_s)ds\big\}}{\ee^x\{\tau_n\}}
\leq (\|f\| + 1) \sup_{\tau} \frac{\ee^x\big\{\int_0^{\tau}\ind{B_{z,N}^c}(X_s)ds\big\}}{\ee^x\{\tau\}}.
\end{align*}
By assumption \ref{ergodicstopping} the right-hand side converges to $0$ as $N \to \infty$. Since the bound does not depend on $V$, this implies that $\hat \lambda_N$ converges, as $N \to \infty$, to $\hat\lambda(x)$. This also implies that $\hat \lambda(x)$ does not depend on $x$.

% Using assumption \ref{ergodicstopping}, for $\ve>0$ there is $N(\ve)$ such that for $N\geq N(\ve)$, $x\in U$ and all stopping times $\tau$ such that $\ee^x\left\{\tau\right\}<\infty$ we have
% \[\ee^x\left\{\int_0^\tau \ind{B_{z,N}^c}(X_s)ds\right\}\leq \ve \ee^x\left\{\tau \right\}.\]
%
% Looking at $\ve$-optimal strategies for functionals $\hat{J}^{N,c}\big(x, V\big)$ and $ \hat{J}\big(x, V\big)$ for $\ve \le 1$ we can restrict ourselves to strategies satisfying
% \begin{equation}
% \limsup_{n\to \infty}\frac{\ee^x\left\{\tau_n\right\}}{n} \leq \frac{\|f\|+2}{-c}.
% \end{equation}
% Therefore for $N\geq N(\ve)$
% \[0\leq \hat{J}^{N,c}\big(x, V\big)- \hat{J}\big(x, V\big)\leq \limsup_{n\to \infty}\frac{\ee^x\left\{\int_0^{\tau_n}(\|f\|+1)\ind{B_{z,N}^c}(X_s)ds\right\}}{\ee^x\left\{\tau_n\right\}}
% \leq \ve, \]
% which completes the proof as $\ve$ can be chosen arbitrarily small.
\end{proof}

\bibliographystyle{amsplain}

\bibliography{references-undis}

\end{document}